\documentclass[11pt]{article}%
\usepackage[latin1]{inputenc}
\usepackage{srcltx}
\usepackage{latexsym}
\usepackage{amsthm}
\usepackage[Sonny]{fncychap}
\usepackage{enumerate}
\usepackage[dvips]{color,graphicx}
\usepackage[dvipdfm]{hyperref}
\usepackage{pst-all}
\usepackage{pstricks}
\usepackage{amsbsy}
\usepackage{amsmath}
\usepackage{amsfonts}
\usepackage{amssymb}
\usepackage{graphicx}%
\newtheorem{theorem}{Theorem}[section]
\newtheorem{definition}[theorem]{Definition}

\newtheorem{corollary}[theorem]{Corollary}
\newtheorem{lemma}[theorem]{Lemma}

\newtheorem*{pft3}{\bf{Proof of Theorem \ref{enteor31}}}

\newtheorem*{pfc}{\bf Proof of Corollary \ref{encorola3}}
\newtheorem*{333pfc}{\bf Proof of Corollary \ref{3corola3}}
\newtheorem*{pft1}{\bf Proof of Theorem \ref{enteorema2}}
\newtheorem*{33pff}{\bf Proof of Theorem \ref{3enteo2}}
\newtheorem*{pft11}{\bf Proof of Theorem \ref{enteor1}}
\numberwithin{equation}{section}

\begin{document}

\title{Stochastically Perturbed Chains of Variable Memory}
\author{{{Nancy L. Garcia} {\thanks{Rua S\'ergio Buarque de Holanda,
        651, CEP 13083-859, Campinas-SP, Brazil. \ E-mail: \ nancy@ime.unicamp.br}}}\\{\small Universidade Estadual de Campinas, Departamento de Estat\'{\i}stica.} \vspace{0.5cm}\\{{Lucas Moreira}{\thanks{Campus Universit\'ario Darcy Ribeiro
      ICC Centro - Bloco A - Asa Norte, CEP 70910-900, Bras\'{\i}lia-DF, Brazil.\ E-mail: \ lmoreira@unb.br}}}\\{\small Universidade de Bras\'{\i}lia, Departamento de Estat\'{\i}stica.} }
\maketitle
\date{}

\begin{abstract}
  In this paper, we study inference for chains of variable order under
  two distinct contamination regimes. Consider we have a chain of
  variable memory on a finite alphabet containing zero. At each
  instant of time an independent coin is flipped and if it turns head
  a contamination occurs. In the first regime a zero is read
  independent of the value of the chain. In the second regime, the
  value of another chain of variable memory is observed instead of the
  original one. Our results state that the difference between the
  transition probabilities of the original process and the
  corresponding ones of the contaminated process may be bounded above
  uniformly.  Moreover, if the contamination probability is small
  enough, using a version of the Context algorithm we are able to
  recover the context tree of the original process through a
  contaminated sample.  \bigskip



\end{abstract}

\section{Introduction}

The goal of this paper is to answer the question proposed by Collet,
Galves and Leonardi (2008): ``Is it possible to recover the context
  tree of a variable length Markov chain from a noisy sample of
  the chain?'' In this paper, we answer positively this question for
two distinct contamination models.

Unbounded variable length Markov chains define a very flexible class
of stochastic chains of infinite order on a finite alphabet.  The idea
is that for each past, only a finite suffix of the past, called {\sl
  context}, is enough to predict the next symbol. These suffixes can be
represented by a countable, complete tree of finite contexts called
{\it context tree}. In a probabilistic suffix tree there is a
transition probability associated to each context.

Probabilistic suffix trees were first introduced by Rissanem (1983) in
the finite case as a flexible and parsimonious modelization tool for
data compression, approximating Markov chains of finite orders. He
called his model {\it finitely generated source}. In his work, not
only he introduces the model but also he proposes the algorithm {\sl
  Context} to estimate the context needed to predict the next
symbol, given a finite sample in an effective way. These models became
popular in the statistics literature under the name {\it Variable
  Length Markov Chains} coined by  B\"uhlmann and Wyner (1999). We
refer the reader to Galves and L\"ocherbach (2008) for a detailed
review on the subject.

In this paper we will study two contamination regimes for chains of
infinite order on a finite alphabet. In the first regime, for
simplicity, we assume the alphabet to be binary and at each step the
symbol 1 turns into a 0 with a small fixed probability independently
of everything. In the second regime, consider that we have another
chain of infinite order with the same alphabet as the original one. At
each instant of time the process randomly chooses the contaminant
process over the original one with a small fixed probability.

Our results state that the difference between the conditional
probabilities of the original process and the corresponding ones of
the contaminated process may be limited above uniformly. Furthermore,
we show that this upper bound is an increasing function of the
contamination probability.

Using a variant of the algorithm Context presented in Galves and
Leonardi (2008), our first result proves that even though a
contaminated sample was used, the estimated tree recovers the context
tree of the original process. That is, the proposed estimator of the context
tree is robust. 

We also observed that the results obtained for the first regime can be
easily extended to chains taking values on a finite size alphabet. 

Our paper is organized as follows, Section \ref{defres111} presents
some basic definitions. Section \ref{models888} presents the
contamination regimes and our main results. Section \ref{todasdem} is
dedicated to the proof of the results. Section \ref{comparacao}
compares the results presented in this work with the corresponding one
presented in Collet, Galves and Leonardi (2008).

\section{Definitions}\label{defres111}

Without loss of generality, let us consider the alphabet
${\cal{A}}\,=\,\{0,1,\ldots, N-1\}$, with size $|{\cal{A}}| =
N$. 

Given two integers $m\leq n$, we denote by $a_m^n$ the string $a_m
\ldots a_n$ of symbols in $A$. For any $m\leq n$, the length of the
string $a_m^n$ is denoted by $l(a_{m}^{n})$ and defined by $n-m+1$. We
will often use the notation $\emptyset$ which will stand for the empty
string, having length $|\emptyset|=0$. For any $n\in\mathbb{Z}$, we
will use the convention that $a_{n+1}^{n}=\emptyset$, and naturally
$l(a_{n+1}^{n})=0$. Given two strings $v$ and $v'$, we denote by $vv'$
the string of length $l(v) + l(v') $ obtained by concatenating the two
strings. If $v'=\emptyset$, then $v\emptyset=\emptyset v=v$. The
concatenation of strings is also extended to the case where $v=\ldots
a_{-2}a_{-1}$ is a semi-infinite sequence of symbols. If
$n\in\{1,2,\ldots\}$ and $v$ is a finite string of symbols in $A$,
$v^{n}=v\ldots v$ is the concatenation of $n$ times the string $v$. In
the case where $n=0$, $v^{0}$ is the empty string $\emptyset$. We say
that the sequence $s$ is a {\it suffix} of the sequence $\omega$ if there
exists a sequence $u$, with $l(u)\geq 1$, such that $\omega = us$. In
this case we write $s\prec \omega$. When $s\prec \omega$ or $s =
\omega$ we write $s\ \underline{\prec} \ \omega$. Given a finite
sequence $\omega$ we denote by $\mbox{suf}(\omega)$ the largest suffix
of $\omega$. Let
$$
{\cal A}_{-\infty}^{-1}={\cal A}^{\{\ldots,-2,-1\}}\,\,\,\,\,\,\textrm{ and
}\,\,\,\,\,\,\, {\cal A}^{\star} \,=\,
\bigcup_{j=0}^{+\infty}\,A^{\{-j,\dots, -1\}}\, ,
$$
be, respectively, the set of all infinite strings of past symbols and
the set of all finite strings of past symbols.  The case $j=0$
corresponds to the empty string $\emptyset$. Finally, we denote by
$\underline{a}=\ldots a_{-2}a_{-1}$ the elements of $A_{-\infty}^{-1}$.

Throughout this paper, we consider ${\bf X} = \{X_t, \, t \in \mathbb{Z} \}$ and
${\bf Y}=\{Y_t, \, t \in \mathbb{Z} \}$ stationary ergodic stochastic processes
over the same finite alphabet ${\cal{A}}$. Given two sequences
$\omega, \ v \in {\cal A}_{-\infty}^{-1}$ and symbol $a, \ b \in {\cal{A}}$,
let
$$
p_{X}(a\left|\right.\omega):=\mathbb{P}(X_{0}=a\left|\right. X_{-1}=\omega_{-1}, X_{-2}=\omega_{-2},\ldots),
$$
$$
p_{Y}(b\left|\right.v):=\mathbb{P}(Y_{0}=b\left|\right. Y_{-1}=v_{-1}, Y_{-2}=v_{-2},\ldots),
$$
that is, the {\bf X} and {\bf Y} processes are compatible with the
transition probabilities $p_{X}(\cdot|\cdot)$ and
$p_{Y}(\cdot|\cdot)$, respectively. Given two finite sequences
$\omega$, $v \in {\cal{A}}_{-j}^{-1}$ we denote by 
$$
\mu_{X}(\omega):=\mathbb{P}\left(X_{-j}^{-1}=\omega\right), \,\,\,\,\,\,\,\,\,
\mu_{Y}(v):=\mathbb{P}\left(Y_{-j}^{-1}=v\right),
$$
the stationary probabilities of the  cylinders  defined by the sequences $\omega$ and $v$, respectively.


\begin{definition} \label{nonnull}
\begin{description}
\item(1){ {\rm{Non-nullness.}} A process {\bf X} is said to be
    non-null if it satisfies
$$
\alpha_{X}:=\inf\left\{p_{X}(a\left|\right.\omega):a\in {\cal{A}}, \ \omega \in {\cal{A}}_{-\infty}^{-1}\right\}>0.
$$}

\item(2){ {\rm{Summable continuity rate.}} A process {\bf X} has
    summable continuity rate if
$$
\beta_{X}:=\displaystyle\sum_{k\in \mathbb{N}}{\beta_{k, X}}<\infty
$$
\noindent where the sequence $\left\{\beta_{k,X}\right\}_{k \in \mathbb{N}}$ is defined by
\begin{equation*}\label{defbk}
\beta_{k, X}:=\sup\left\{\left|1-\frac{p_{X}(a\left|\right.\omega)}{p_{X}(a\left|\right.v)}\right|:a \in {\cal{A}}, \ v, \ \omega \in {\cal{A}}_{-\infty}^{-1} \ \ with \ \  \omega \overset{k}{=} v \right\}.
\end{equation*}
Here, $\omega \overset{k}{=} v $ means that $\omega_{-k}^{-1} =
v_{-k}^{-1}$. The sequence $\left\{\beta_{k, X}\right\}_{k \in
  \mathbb{N}}$ is called {\rm{continuity rate}} of the process {\bf
  X}.}\label{taxsom}
\end{description}
\end{definition}

\vspace{.2cm}


\begin{definition}\label{sufixo111}
A sequence $\omega \in {\cal{A}}_{-j}^{-1}$ is a {\rm context} for the process {\bf X}  if it satisfies
\begin{description}
	\item(1){ For all semi-infinite sequence $x_{-\infty}^{-1}$ having $\omega$  as a suffix,
\begin{equation}\label{suf1}
\mathbb{P}\left(X_{0}=a\left|\right. X_{-\infty}^{-1}=x_{-\infty}^{-1}\right)=p_{X}(a\left|\right.\omega), \ \mbox{for all} \ a \in {\cal{A}}.
\end{equation}}
\item(2){ No suffix of $\omega$ satisfies (\ref{suf1}).}
\end{description}

An infinite context  is a semi-infinite sequence $\omega_{-\infty}^{-1}$ such that none of its suffixes $\omega_{-j}^{-1}$, $j=1,2,\ldots$ is a context.
\end{definition}

It is easy to see that the set of all contexts (finite or infinite)
can be identified with the set of leaves of a rooted tree with a
countable set of finite labeled branches. This tree is called
{\it{context tree}} of the process {\bf X} and it will be denoted by
${\cal{T}}_{X}$. 

\begin{definition}\label{truncated}
Given an integer $K$, define 
{\rm the
tree  ${\cal{T}}_{X}$ truncated at level  $K$} by 
\begin{equation*}
{{\cal{T}}_{X}\left|\right.}_{K}:=\left\{ \omega \in {\cal{T}}_{X}: l(\omega)\leq K \right\} \cup \left\{\omega:l(\omega)= K  \mbox{ and } \omega \prec u, \ \mbox{for some} \ u \in {\cal{T}}_{X}\right\}.
\end{equation*}
\end{definition}

Given an integer  $k\geq 1$, define
\begin{equation}\label{defCk}
{\cal{C}}_{k}:=\left\{u \in {\cal{T}}_{X}\left|\right. _{k}: p_{X}(a \left|\right. u)\neq p_{X}(a\left|\right. \mbox{suf}(u)),\mbox{for some a} \in {\cal{A}}\right\}
\end{equation}
and
\begin{equation}\label{defDk}
{{D}}_{k}:= \underset{u\in \ C_{k}}{\min} \ \underset{a\in {\cal{A}}}{\max} \left\{\left| p_{X}(a \left|\right. u) - p_{X}(a\left|\right.\mbox{suf}(u)) \right|\right\}.
\end{equation}
From the definition, we can see that ${{D}}_{k} > 0$ for all $k\geq 1$.

\paragraph{Algorithm Context} 

Generically, let {\bf Z} be an infinite order process on the alphabet
${\cal{A}}$. We assume that {\bf Z} is compatible with a transition
kernel $p_ {Z}(\cdot|\cdot)$. Later, the process {\bf Z} will denote
both contamination models studied in this work. Let
$Z_{1},Z_{2},\ldots,Z_{n}$ be a random sample of the process {\bf
  Z}. For any finite sequence $\omega$, with $l(\omega)\leq n$, we
denote by $N_{n}(\omega)$ the number of occurrences of $\omega$ in the
sample, that is
\begin{equation}\label{defnumocoamos}
N_{n}(\omega)=\displaystyle\sum_{t=0}^{n-l(\omega)}{\textbf{1}}_{\left\{Z_{t+1}^{t+l(\omega)} = \ \omega\right\}}.
\end{equation}

For any element $a \in {\cal{A}}$ and for any finite sequence $\omega$, the empirical transition kernel  $\hat{p}_{Z}(a\left|\right.\omega)_n$ is defined by
$$
\hat{p}_{Z}(a\left|\right.\omega)_n=\frac{N_{n}(\omega a)+1}{N_{n}(\omega \cdot)+\left|{\cal{A}}\right|}, \,\,\,\,\,\,\mbox{where}\,\,\,\,\,\,
N_{n}(\omega\cdot)=\displaystyle\sum_{b\in {\cal{A}}}N_{n}(\omega b).
$$

A modification of Rissanen's context tree estimator proposed in Galves
and Leonardi (2008) which we will use in this work is given
below. First, let us define the operator
\begin{equation*}\label{defdelta}
\Delta_{n}(\omega):= \underset{a\in {\cal{A}}}{\max }\left|\hat{p}_Z(a\left|\right.\omega)_n-\hat{p}_Z(a\left|\right.\mbox{suf}(\omega))_n\right|
\end{equation*}
for any finite sequence $\omega$.

\begin{definition}
\label{estarvcont}
(Galves and Leonardi, 2008).
For any $\delta > 0$ and $d<n$, the  context tree estimator
 $\hat{{\cal{T}}}_{n}^{\delta , d}$  is the set containing all sequences $\omega\in {\cal{A}}_{-d}^{-1}$, such that $\Delta_{n}(a  \mbox{suf}(\omega))>\delta$ for some $a\in {\cal{A}}$ and $\Delta_{n}(u\omega)\leq\delta$ for any $u \in {\cal{A}}_{-d}^{-l(\omega)}$.
\end{definition}

\section{Contamination Regimes and Results}\label{models888}

\subsection{Zero inflated contamination}
Initially, we consider {\bf X} a stationary process compatible with
$p_{X}(\cdot|\cdot)$ but taking values in the binary alphabet
${\cal{A}}=\{0,1\}$. Let ${\boldsymbol \xi} = \{ \xi_t, t \in
\mathbb{Z}\}$ be a sequence of i.i.d. Bernoulli random variables taking
values on $\left\{0,1\right\}$, independent of the process {\bf X},
with
$$
\mathbb{P}\left(\xi_{t}=1\right)=1-\varepsilon,  
$$
where $\varepsilon$ is a noise parameter fixed  in $(0,1)$.

We define the Zero inflated contamination model by
\begin{equation}\label{defproczt}
 Z_t\,:=\,  X_t \cdot \xi_t, \quad t \in {\mathbb Z}.
\end{equation}
It is easy to see that {\bf Z}  will be an infinite order process even if the process ${\bf X}$ is a Markov chain of order one.


\begin{theorem}\label{enteor1}
If  {\bf X} is non-null and has summable continuity rate and {\bf Z} is defined by (\ref{defproczt}) then, for any $\varepsilon \in (0,1)$, we have
\begin{equation}\label{teo1111111}
\sup_{k\ge 1} \sup_{\omega_{-k}^{-1}\in {\cal{A}}_{-k}^{-1} }\sup_{a\in {\cal{A}}}{\left|p_{Z}\left(a|\omega_{-k}^{-1}\right)-p_{X}\left(a|\omega_{-k}^{-1}\right) \right| \leq \varepsilon\left[1 + \frac{4\beta_{X}}{\min{(1,(1+\varepsilon)\alpha_{X}\beta_{X}^{*})}}\right]},
\end{equation}
where $\beta_{X}^{*}=\prod_{k=0}^{+\infty}(1-\beta_{k, X})>0.$
\end{theorem}

In order to recover the truncated context tree  process {\bf X} by using a sample of the perturbed process  {\bf Z}, we establish the second result of this section.

\begin{theorem}\label{enteorema2}
Let $K$ be an integer and consider $Z_{1},Z_{2},\ldots,Z_{n}$  a random sample of the perturbed process  {\bf Z}. Then, there exists a constant $c_{1}$ and an integer $d$ depending on the process  {\bf X} such that for any $\epsilon \in \left(0,D_{d}/2c_{1}\right)$, any $\delta \in (c_{1}\epsilon, D_{d}-c_{1}\epsilon)$, there exists $n_{0}(\delta)$ such that for any $n> n_{0}$ we have
$$
\mathbb{P}\left(\hat{{\cal{T}}}_{n}^{\delta, d}\left|\right._{K}\neq {\cal{T}}_{X}\left|\right._{K}\right)\leq c_{2}exp\left\{-c_{3}(n-d)\right\},
$$
where $\hat{{\cal{T}}}_{n}^{\delta , d}$ is as in Definition \ref{estarvcont}. The constants are explicit and given by
\begin{description}
\item(1)  $c_{1}\,=\,2\left[1+\frac{4\beta_{X}}{min(\alpha_{X}\beta_{X}^{*},1)}\right],$ $c_{2}\,=\,2^{d}12e^{\frac{1}{e}}$, $c_{3}\,=\,\frac{\left[\min\left(D_{d} - \delta,\delta\right)-2\bar{k}\right]^{2}{\alpha_{X}^{2d}(1-\varepsilon)^{3d+1}}}{256e(d+1)\left(1+\frac{\beta_{X}}{\alpha_{X}}\right)},$
\item(2)  $d\,=\,\max_{u \notin {\cal{T}}_{X}, \ l(u)< K} \, \min\left\{k: \mbox{there exists} \ \  \omega \in {\cal{C}}_{k} \ \ \mbox{with} \ \ u\prec\omega\right\},$
\item(3)  $n_{0}\,=\,\frac{6}{\left\{D_{d} - \delta -2\varepsilon\left[1 + \frac{4\beta_{X}}{\min{(1,(1+\varepsilon)\alpha_{X}\beta_{X}^{*})}}\right]\right\}\alpha_{X}^{d}(1-\varepsilon)^{d}}+d,$ 
\item(4) $\bar{k}\,=\, \varepsilon\left[1 + \frac{4\beta_{X}}{\min{(1,(1+\varepsilon)\alpha_{X}\beta^{*})}}\right]+\frac{3}{(n-d)\alpha_{X}^{d}(1-\varepsilon)^{d}}\cdot$
\end{description}
\end{theorem}


\begin{corollary}\label{encorola3}
For all integer $K$ and for almost all infinite sample $Z_{1},Z_{2},\ldots$ there exists an $\overline{n}$ such that, for any $n\geq\overline{n}$ we have
$$
\hat{{\cal{T}}}_{n}^{\delta, d}\left|\right._{K}={\cal{T}}_{X}\left|\right._{K},
$$
where $d$ and $\delta$ are chosen as in Theorem \ref{enteorema2}. 
\end{corollary}

\subsection{Process contamination}
Let {\bf X} and {\bf Y} be independent processes taking values on the
alphabet ${\cal{A}}=\left\{0,1,\ldots,N-1\right\}$ and compatible with
the transition probabilities $p_{X}(\cdot|\cdot)$ and
$p_{Y}(\cdot|\cdot)$ respectively. Furthermore, we suppose that these
processes  are non-null and have summable
continuity rate with constantes $\alpha_{X}$, $\beta_{X}$, $\alpha_{Y}$ and $\beta_{Y}$, respectively. Let ${\boldsymbol \xi} = \{\xi_t, t \in {\mathbb Z}\}$  be a sequence of i.i.d. random variables  taking values on $\left\{0,1\right\}$, independent of the processes {\bf X} and {\bf Y}, with 
$$
\mathbb{P}\left(\xi_{t}=1\right)=1-\varepsilon,  
$$
where $\varepsilon$ is a parameter fixed in $(0,1)$. Define the process {\bf Z} by
\begin{center}
\begin{eqnarray}\label{lulu1} Z_t = \left\{
\begin{array}{rl}
X_t, &\mbox{if}\ \xi_t=1,\\
Y_t,& \mbox{if}\ \xi_t=0 \\
\end{array}
\right.\end{eqnarray}
\end{center}
for all $t \in {\mathbb Z}$. In fact, the contamination we propose is
that at each time the process {\bf Z} chooses with probability
$1-\epsilon$ and $\epsilon$ to use the transition law of {\bf X} or
{\bf Y} independently of everything. Therefore, model (\ref{lulu1})
corresponds to
\begin{center}
\begin{eqnarray}\label{lulu1a} p_{Z}(\cdot|\omega) = \left\{
\begin{array}{rl}
p_{X}(\cdot|\omega), &\mbox{if}\ \xi_t=1,\\
p_{Y}(\cdot| \omega),& \mbox{if}\ \xi_t=0 
\end{array}
\right.\end{eqnarray}
\end{center}

Generically, {\bf Z} is an infinite order process and it can be
interpreted as a stochastic perturbation of the process {\bf X} if
$\varepsilon$ is sufficiently small. For this model, we obtained
similar results to the zero inflated contamination.


\begin{theorem}\label{enteor31}
Let {\bf X} and {\bf Y} be independent processes which are non-null and have summable continuity rate and {\bf Z} defined by (\ref{lulu1}). Then, for all $\varepsilon \in (0,1)$,  we have 
\begin{equation*}\label{3teo1}
\sup_{k\ge 1}\sup_{a \in {\cal{A}}} \sup_{\omega_{-k}^{-1}  \in {\cal{A}}_{-k}^{-1}}{\left|p_{Z}\left(a|\omega_{-k}^{-1}\right)-p_{X}\left(a|\omega_{-k}^{-1}\right) \right| \leq \varepsilon\left[2 + \frac{4(N-1)\beta_{X}}{\min{(1,\alpha\beta^{*}_{min})}}\right]},
\end{equation*}
where  $\beta_{X}^{*}=\prod_{k=0}^{+\infty}(1-\beta_{k,X})>0$ and $\alpha\beta^{*}_{min}=\min\{\alpha_{X}\beta_{X}^{*},\alpha_{Y}\}$.
\end{theorem}

\vspace{.2cm}

\begin{theorem}\label{3enteo2}
Let $K$ be an integer and let $Z_{1},Z_{2},\ldots,Z_{n}$ be a random sample of the process {\bf Z}. Then, there exists a constant $c_{1}$ depending on the processes {\bf X} and {\bf Y}  and there exists an integer $d$ depending on the process {\bf X} such that for any $\epsilon \in \left(0,D_{d}/2c_{1}\right)$, any  $\delta \in (c_{1}\epsilon, D_{d}-c_{1}\epsilon)$, there exists $n_{0}(\delta)$  such that for all $n> n_{0}$ we have
$$
\mathbb{P}\left(\hat{{\cal{T}}}_{n}^{\delta, d}\left|\right._{K}\neq {\cal{T}}_{X}\left|\right._{K}\right)\leq c_{2}exp\left\{-c_{3}(n-d)\right\},
$$
where $\hat{{\cal{T}}}_{n}^{\delta , d}$ is as in Definition \ref{estarvcont}. All the constants above are explicit and given by
\begin{description}
\item(1)  $c_{1}\,=\,4\left[1+\frac{2(N-1)\beta_{X}}{min((\alpha\beta^{*})_{min},1)}\right],$ $c_{2}\,=\,48N^{d}(N+1)e^{\frac{1}{e}}$, $c_{3}\,=\,\frac{\left[\min\left(D_{d} - \delta,\delta\right)-2\bar{k}\right]^{2}{\alpha^{2d}}}{128N^{2}e(d+1){\beta}_{\alpha, \ max}},$ 
\item(2)  $d\,=\,\max_{u \notin {\cal{T}}_{X}, \ l(u)< K} \min\left\{k: \mbox{there exists} \ \  \omega \in {\cal{C}}_{k} \ \ \mbox{with} \ \ u\prec\omega\right\},$
\item(3)  $n_{0}\,=\,\frac{2(N+1)}{\left\{D_{d} - \delta -4\varepsilon\left[1 + \frac{2(N-1)\beta_{X}}{\min{(1,(\alpha\beta^{*})_{min})}}\right]\right\}\alpha_{min}^{d}}+d,$ $\bar{k}\,=\, \frac{\varepsilon c_{1}}{2}+\frac{N+1}{(n-d)\alpha_{min}^{d}},$
\item(4) $\alpha_{min}\,=\,\min \{\alpha_{X},\alpha_{Y}\}$, ${\beta}_{\alpha, \max}\,=\,\min\left\{\left(1+\frac{\beta_{X}}{\alpha_{X}}\right),\left(1+\frac{\beta_{Y}}{\alpha_{Y}}\right)\right\}\cdot$
\end{description}
\end{theorem}


\begin{corollary}\label{3corola3}
For all integer $K$ and for almost all infinite sample $Z_{1},Z_{2},\ldots$ there exists an $\overline{n}$ such that for all $n\geq\overline{n}$ we have
$$
\hat{{\cal{T}}}_{n}^{\delta, d}\left|\right._{K}={\cal{T}}_{X}\left|\right._{K},
$$
where $d$ and $\delta$ are the same as in Theorem \ref{3enteo2}. 
\end{corollary}

\vspace{.2cm}

\paragraph{Robustness}
In Theorems \ref{enteorema2} and \ref{3enteo2} as well as in Corollaries \ref{encorola3}  and \ref{3corola3} we can see that the estimator $\hat{{\cal{T}}}_{n}^{\delta, d}\left|\right._{K}$ is robust. Here robustness means that even if the estimation process is based on a random sample of the  perturbed process {\bf Z}, the estimator $\hat{{\cal{T}}}_{n}^{\delta, d}\left|\right._{K}$ is able to recover the truncated context tree  ${\cal{T}}_{X}\left|\right._{K}$ of the original process {\bf X}.

\section{Proofs}\label{todasdem}

\subsection{Proof of Theorem \ref{enteor1}}

In order to prove the Theorem  \ref{enteor1} we establish three lemmas.

\begin{lemma}\label{lema4}
  For any $\varepsilon \in (0,1)$, any $\omega_{-\infty}^{0} \in {\cal
    A}_{-\infty}^{-1}$, any $k>j\geq 0$ and any $a,b \in {\cal{A}}$,
  we have
$$
\left|\mathbb{P}\left(X_{0}=\omega_{0}|X_{-j}^{-1}=\omega_{-j}^{-1}, X_{-j-1}=a, Z_{-j-1}=b, Z_{-k}^{-j-2}=\omega_{-k}^{-j-2}\right)-p_{X}(\omega_{0}|\omega_{-\infty}^{-1}) \right|\leq\beta_{j,X}.
$$
\end{lemma}

\begin{proof}
For all $j\geq 0$, we can write the following identity  
\begin{eqnarray}
\lefteqn{\mathbb{P}\left(X_{0}=\omega_{0}|X_{-j}^{-1}=\omega_{-j}^{-1}, X_{-j-1}=a, Z_{-j-1}=b, Z_{-k}^{-j-2}=\omega_{-k}^{-j-2}\right)}\label{conl1}\\
&=&\frac{\displaystyle\sum_{u_{-k}^{-j-2}}p_{X}\left(u_{-k}^{-j-2}a\omega_{-j}^{-1}\omega_{0}\right)\mathbb{P}\left(Z_{-k}^{-j-1}=\omega_{-k}^{-j-2}b|X_{-k}^{-j-1}=u_{-k}^{-j-2}a\right)}{\displaystyle\sum_{u_{-k}^{-j-2}}p_{X}\left(u_{-k}^{-j-2}a\omega_{-j}^{-1}\right)\mathbb{P}\left(Z_{-k}^{-j-1}=\omega_{-k}^{-j-2}b|X_{-k}^{-j-1}=u_{-k}^{-j-2}a\right)},\nonumber
\end{eqnarray}
where the last two sums are over all sequences $u_{-k}^{-j-2}\in{\cal{A}}_{-k}^{-j-2}$ such that
$$
\mathbb{P}\left(X_{-k}^{-j-1}=u_{-k}^{-j-2}a,Z_{-k}^{-j-1}=\omega_{-k}^{-j-1}b\right)\neq 0.
$$
Now, as in Fernandez and Galves (2002), we have
\begin{equation}\label{eq31}
\inf_{v_{-\infty}^{-j-1}}p_{X}\left(\omega_{0}|v_{-\infty}^{-j-1}\omega_{-j}^{-1}\right)\leq p_{X}\left(\omega_{0}|u_{-k}^{-j-2}a\omega_{-j}^{-1}\right)\leq \sup_{v_{-\infty}^{-j-1}}p_{X}\left(\omega_{0}|v_{-\infty}^{-j-1}\omega_{-j}^{-1}\right).
\end{equation}
Since the process {\bf X} is continuous, by (\ref{eq31}), it follows that
\begin{equation}\label{eq32}
p_{X}\left(\omega_{0}|\omega_{-\infty}^{-1}\right)-\beta_{j,X}\leq p_{X}\left(\omega_{0}|u_{-k}^{-j-2}a\omega_{-j}^{-1}\right)\leq p_{X}\left(\omega_{0}|\omega_{-\infty}^{-1}\right)+\beta_{j,X}.
\end{equation}
Thus, by plugging (\ref{conl1}) into (\ref{eq32}),  we conclude the proof of this lemma.
\end{proof}

\begin{lemma}\label{enlema5}
For any $\varepsilon \in (0,1)$, any $k\geq 0$ and any
$\omega_{-k}^{-1} \in {\cal A}_{-k}^{-1}$, $\omega_0 \in {\cal A}$, we have
\begin{equation}\label{eq2.1}
p_{Z}\left(\omega_{0}|\omega_{-k}^{-1}\right)\geq (1-\varepsilon)\alpha_{X}
\end{equation}
and
\begin{equation}\label{eq2.2}
\mathbb{P}\left(X_{0}=\omega_{0}|Z_{-k}^{-1}=\omega_{-k}^{-1}\right)\geq\alpha_{X}.
\end{equation}
Furthermore, for any $0\leq j\leq k$ we can write
\begin{equation}\label{eq2.3}
\mathbb{P}\left(X_{-j-1}=\omega_{-j-1}|X_{-j}^{-1}=\omega_{-j}^{-1}, Z_{-k}^{-j-2}=\omega_{-k}^{-j-2}\right)\geq \alpha_{X}\beta_{X}^{*}.
\end{equation}
\end{lemma}
\begin{proof}
  Initially, we will prove that (\ref{eq2.1}) follows from
  (\ref{eq2.2}). To see that, let us first consider the case
  $\omega_{0}= 0$, we have
\begin{eqnarray*}
  p_{Z}\left(0|\omega_{-k}^{-1}\right)
  &=&\varepsilon\mathbb{P}\left(X_{0}=0|Z_{-k}^{-1}=\omega_{-k}^{-1}\right)+(1-\varepsilon)\mathbb{P}\left(X_{0}=0|Z_{-k}^{-1}=\omega_{-k}^{-1}\right) + \\
  & & + \varepsilon\mathbb{P}\left(X_{0}=1|Z_{-k}^{-1}=\omega_{-k}^{-1}\right)\geq (1+\varepsilon)\alpha_{X}.
\end{eqnarray*}
Similarly, for the case $\omega_{0} = 1$,  we can write
\begin{eqnarray*}
p_{Z}\left(1|\omega_{-k}^{-1}\right)=\mathbb{P}\left(X_{0}=1,\xi_{0}=1|Z_{-k}^{-1}=\omega_{-k}^{-1}\right)\geq (1-\varepsilon)\alpha_{X}.
\end{eqnarray*}
Therefore, for any $\omega_{0} \in {\cal{A}}$ we obtain $
p_{Z}\left(\omega_{0}|\omega_{-k}^{-1}\right)\geq (1-\varepsilon)\alpha_{X}.
$
Now, it is easy to see that
\begin{eqnarray*}
\lefteqn{\mathbb{P}\left(X_{0}=\omega_{0}|Z_{-k}^{-1}=\omega_{-k}^{-1}\right)}\\
&=&\lim_{l\rightarrow +\infty}\frac{\displaystyle\sum_{u_{-l}^{-1}}p_{X}\left(\omega_{0}|\omega_{-\infty}^{-l-1}u_{-l}^{-1}\right)p_{X}\left(u_{-l}^{-1}|\omega_{-\infty}^{-l-1}\right)\mathbb{P}\left( \displaystyle\bigcap_{-k \leq -t \leq -1: \ u_{-t}\neq 0}\xi_{-t}=w_{-t}\right)}{\displaystyle\sum_{u_{-l}^{-1}}p_{X}\left(u_{-l}^{-1}|\omega_{-\infty}^{l-1}\right)\mathbb{P}\left( \displaystyle\bigcap_{-k \leq -t \leq -1: \ u_{-t}\neq 0}\xi_{-t}=w_{-t}\right)},
\end{eqnarray*}
where the last sums are over all the sequences $u_{-l}^{-1}\in{\cal{A}}_{-l}^{-1}$ such that 
$$\mathbb{P}\left(X_{-l}^{-1}=u_{-l}^{-1},Z_{-k}^{-1}=\omega_{-k}^{-1}\right)\neq 0.$$ From the non-nullness hypothesis of the process  {\bf X}, we can write $p_{X}\left(\omega_{0}|\omega_{-\infty}^{-l-1}u_{-l}^{-1}\right)\geq\alpha_{X}.$ Consequently,
$$
\mathbb{P}\left(X_{0}=\omega_{0}|Z_{-k}^{-1}=\omega_{-k}^{-1}\right)\geq\alpha_{X}.
$$
In order to show (\ref{eq2.3}) we observe that the following equality is true
\begin{eqnarray}
\lefteqn{\mathbb{P}\left(X_{-j-1}=\omega_{-j-1}|X_{-j}^{-1}=\omega_{-j}^{-1}, Z_{-k}^{-j-2}=\omega_{-k}^{-j-2}\right)}\label{eq2.2888}\\
&=&\frac{\displaystyle\sum_{x_{-k}^{-j-2}}\mathbb{P}\left(Z_{-k}^{-j-2}=\omega_{-k}^{-j-2}|X_{-k}^{-j-2}=x_{-k}^{-j-2}\right)p_{X}\left(x_{-k}^{-j-2}\omega_{-j-1}^{-1}\right)}{\displaystyle\sum_{x_{-k}^{-j-2}}\mathbb{P}\left(Z_{-k}^{-j-2}=\omega_{-k}^{-j-2}|X_{-k}^{-j-2}=x_{-k}^{-j-2}\right)p_{X}\left(x_{-k}^{-j-2}\omega_{-j}^{-1}\right)},\nonumber
\end{eqnarray}
where the the last sums are over all sequences $x_{-k}^{-j-2}\in{\cal{A}}_{-k}^{-j-2}$ such that
$$
\mathbb{P}\left(X_{-k}^{-j-2}=x_{-k}^{-j-2},Z_{-k}^{-j-2}=\omega_{-k}^{-j-2}\right)\neq 0.
$$
Now we can see that
\begin{eqnarray}\label{e.eeee}
\frac{p_{X}\left(x_{-k}^{-j-2}\omega_{-j-1}^{-1}\right)}{p_{X}\left(x_{-k}^{-j-2}\omega_{-j}^{-1}\right)}
&\geq&\alpha_{X}\beta_{X}^{*}
\end{eqnarray}
The proof of the third statement of this lemma follows from the last inequality and identity (\ref{eq2.2888}).
\end{proof}

\begin{lemma}\label{lema6}
For any $\varepsilon \in (0,1)$, $k>j\geq 0$ and  $\omega_{-k}^{0} \in {\cal{A}}$, we have
$$
\sup_{j,k}\sup_{\omega_{-k}^{0}}{\mathbb{P}\left(X_{-j-1}={\omega' }_{-j-1}|X_{-j}^{-1}=\omega_{-j}^{-1}, Z_{-k}^{-j-1}=\omega_{-k}^{-j-1}\right)}\leq \frac{\varepsilon}{(1+\varepsilon)\alpha_{X}\beta_{X}^{*}},
$$ 
where ${\omega' }_{-j-1}\neq{\omega }_{-j-1}$.
\end{lemma}
\begin{proof}
For ${\omega' }_{-j-1}\neq{\omega }_{-j-1}$, we have
\begin{eqnarray*}
\lefteqn{\mathbb{P}\left(X_{-j-1}={\omega' }_{-j-1}|X_{-j}^{-1}=\omega_{-j}^{-1}, Z_{-k}^{-j-1}=\omega_{-k}^{-j-1}\right)} \\
&=&\frac{\mathbb{P}\left(X_{-j-1}={\omega}'_{-j-1},Z_{-j-1}=\omega_{-j-1}|X_{-j}^{-1}=\omega_{-j}^{-1}, Z_{-k}^{-j-2}=\omega_{-k}^{-j-2}\right)}{\mathbb{P}\left(Z_{-j-1}=\omega_{-j-1}|X_{-j}^{-1}=\omega_{-j}^{-1}, Z_{-k}^{-j-2}=\omega_{-k}^{-j-2}\right)}\cdot
\end{eqnarray*}
Now, by Lemma \ref{enlema5} and the last equality we can obtain  
$$
\mathbb{P}\left(X_{-j-1}={\omega' }_{-j-1}|X_{-j}^{-1}=\omega_{-j}^{-1}, Z_{-k}^{-j-1}=\omega_{-k}^{-j-1}\right)\leq \frac{\varepsilon}{(1+\varepsilon)\alpha_{X}\beta_{X}^{*}},$$
as desired.

\end{proof}


\begin{pft11} 
{\rm 
For any $a \in {\cal{A}}$ and every sequence $\omega_{-k}^{-1} \in {\cal{A}}_{-k}^{-1}$, one can see that
\begin{equation}\label{t1}
\left|p_{Z}\left(a|\omega_{-k}^{-1}\right) - \mathbb{P}\left(X_{0}=a|Z_{-k}^{-1}=\omega_{-k}^{-1}\right) \right| \leq \varepsilon.
\end{equation}
Observe that, for  $k=0$, the assertion of Theorem \ref{enteor1} is trivially valid. 
So it remains to prove  (\ref{teo1111111}) for $k\geq 1$. In this case, we can write
\begin{eqnarray}
\lefteqn{\mathbb{P}(X_{0}=a|Z_{-k}^{-1}=\omega_{-k}^{-1})-\mathbb{P}(X_{0}=a|X_{-k}^{-1}=\omega_{-k}^{-1})}\nonumber\\
&=&\displaystyle\sum_{j=0}^{k-1}\left[\mathbb{P}(X_{0}=a|X_{-j}^{-1}=\omega_{-j}^{-1},Z_{-k}^{-j-1}=\omega_{-k}^{-j-1})\right.\nonumber \\
& & - \left.\mathbb{P}(X_{0}=a|X_{-j-1}^{-1}=\omega_{-j-1}^{-1},Z_{-k}^{-j-2}=\omega_{-k}^{-j-2}) \right].\label{sumteo11}
\end{eqnarray}
We can show that each parcel of the last sum can be rewritten as follows 
\begin{eqnarray}
\lefteqn{\displaystyle\sum_{b\in
    \left\{0,1\right\}}[\mathbb{P}(X_{0}=a|X_{-j}^{-1}=\omega_{-j}^{-1},X_{-j-1}=b,
  Z_{-k}^{-j-1}= \omega_{-k}^{-j-1})} \nonumber \\
&& - \mathbb{P}(X_{0}=a|X_{-j-1}^{-1}=\omega_{-j-1}^{-1},Z_{-k}^{-j-2}=\omega_{-k}^{-j-2})]\nonumber\\
& & \times \, \mathbb{P}(X_{-j-1}=b|X_{-j}^{-1}=\omega_{-j}^{-1},Z_{-k}^{-j-1}=\omega_{-k}^{-j-1}).\label{rem1ff}
\end{eqnarray}
For each $j$, $0\leq j\leq k-1$, the sum above has two parcels. From Lemma \ref{lema6} the parcel corresponding to $b={\omega'}_{-j-1}$, with ${\omega'}_{-j-1}\neq{\omega}_{-j-1}$, can be upper bounded by
\begin{eqnarray}
2\beta_{j,X}\frac{\varepsilon}{(1+\varepsilon)\alpha_{X}\beta_{X}^{*}}.\label{t7}
\end{eqnarray}
Now, for each $j$,  we shall limit the parcel of (\ref{rem1ff}) corresponding to $b=\omega_{-j-1}$ in (\ref{rem1ff}) by 
\begin{eqnarray}
\lefteqn{\displaystyle\sum_{c\in \left\{0,1\right\}}|\mathbb{P}(X_{0}=a|X_{-j}^{-1}=\omega_{-j}^{-1},X_{-j-1}={\omega}_{-j-1}, Z_{-k}^{-j-1}=\omega_{-k}^{-j-1})}\nonumber\\ 
&-&\mathbb{P}(X_{0}=a|X_{-j-1}^{-1}=\omega_{-j-1}^{-1},Z_{-k}^{-j-2}=\omega_{-k}^{-j-2},Z_{-j-1}=c)|\nonumber\\
&\times&\mathbb{P}(X_{-j-1}=\omega_{-j-1}|X_{-j}^{-1}=\omega_{-j}^{-1},Z_{-k}^{-j-1}=\omega_{-k}^{-j-1})\times\nonumber\\
&\times&\mathbb{P}(Z_{-j-1}=c|X_{-j-1}^{-1}=\omega_{-j-1}^{-1},Z_{-k}^{-j-2}=\omega_{-k}^{-j-2})\label{rem0fgh}
\end{eqnarray}
We observe that the parcel corresponding to $c=\omega_{-j-1}$ of (\ref{rem0fgh}) is null. When $c={\omega'}_{-j-1}$, with ${\omega'}_{-j-1}\neq{\omega}_{-j-1}$, the corresponding parcel of  (\ref{rem0fgh}) can be bounded from above by $2\beta_{j,X}\varepsilon$. From this, (\ref{t7}) and (\ref{t1}) we conclude the proof of Theorem \ref{enteor1}.}
\end{pft11}

\subsection{Proof of Theorem \ref{enteorema2}}\label{KXXX}

The proof of our second main result is based on four lemmas. The first one is consequence of Lemma $3.4$ of Galves and Leonardi (2008).

\begin{lemma}\label{lema88888} 
There exists a summable sequence  $\left(\rho_{l,X}\right)_{l \in \mathbb{N} }$, satisfying 
\begin{equation}\label{galeo2008}
\sum_{l\in \mathbb{N}}\rho_{l,X}\leq 2 \left( 1 + \frac{\beta_{X}}{\alpha_{X}}\right),
\end{equation}
such that for any $i\geq 1$, any $k \geq i$, any $j\geq 1$ and any finite sequence $\omega_{1}^{j}$, the following inequality holds
$$
\sup_{x^{i}_{1}, \theta_{1}^{i}\in {\cal{A}}^{i}}\left|\mathbb{P}\left(Z_{k}^{k+j-1}=\omega^{j}_{1}\left.\right|X_{1}^{i}=x_{1}^{i}, \xi_{1}^{i}=\theta_{1}^{i}\right)-\mu_{Z}\left(\omega^{j}_{1}\right)\right| \leq \frac{\displaystyle\sum_{l = 0}^{j-1}\rho_{k-i+l,X}}{(1-\varepsilon)^{j}},
$$
where $\alpha_{X}$ and $\beta_{X}$ are the same quantities as those in Definition \ref{nonnull}.
\end{lemma}
\begin{proof}
Since the processes {\bf X}, {\bf  Y} and ${\boldsymbol \xi}$ are
independent, we have  
\begin{eqnarray} 
\lefteqn{\left|\mathbb{P}\left({Z_{k}}^{k+j-1}=\omega^{j}_{1}\left.\right|X_{1}^{i}=x_{1}^{i}, \xi_{1}^{i}=\theta_{1}^{i}\right)-\mu_{Z}\left(\omega^{j}_{1}\right)\right|}\nonumber\\
&=&\left|\displaystyle\sum_{x_{k}^{k+j-1}}\mathbb{P}\left( X_{k}^{k+j-1}= x_{k}^{k+j-1}, Z_{k}^{k+j-1}=\omega_{1}^{j}|X_{1}^{i} = x_{1}^{i},\xi_{1}^{i}=\theta_{1}^{i}\right)-\mu_{Z}\left(\omega^{j}_{1}\right)\right| \nonumber\\
&=&\left|\displaystyle\sum_{x_{k}^{k+j-1}}p_{X}\left(x_{k}^{k+j-1}|x_{1}^{i}\right)\mathbb{P}\left(Z_{k}^{k+j-1}=\omega^{j}_{1}|X_{k}^{k+j-1}=x_{k}^{k+j-1}\right) -\mu_{Z}\left(\omega_{1}^{j}\right)\right|,\label{isk44}
\end{eqnarray}
for any $x^{i}_{1}, \theta_{1}^{i}\in {\cal{A}}_{1}^{i}$. The last two summations are over the set 
\begin{equation}\label{kkllljdjdjdd}
{\cal{C}^{*}}=\left\{x_{k}^{k+j-1}\in {\cal{A}}_{1}^{j}\left|\right.\left\{X_{k}^{k+j-1}=x_{k}^{k+j-1},Z_{k}^{k+j-1}=\omega^{j}_{1}\right\}\neq \emptyset\right\}.
\end{equation}
On the other hand, we can see that   
\begin{eqnarray*}
\mu_{Z}\left(\omega^{j}_{1}\right)&=&\displaystyle\sum_{x_{k}^{k+j-1}\in\ {\cal{C}^{*}}}\mathbb{P}\left(Z_{k}^{k+j-1}=\omega^{j}_{1}| X_{k}^{k+j-1}=x_{k}^{k+j-1}\right)\mathbb{P}\left( X_{k}^{k+j-1}= x_{k}^{k+j-1}\right). 
\end{eqnarray*}
Then, by using the last identity, (\ref{isk44}) and the Lemma $3.4$ of Galves and Leonardi (2008), we can write
$$
\left|\mathbb{P}\left({Z_{k}}^{k+j-1}=\omega^{j}_{1}\left.\right|X_{1}^{i}=x_{1}^{i}, \xi_{1}^{i}=\theta_{1}^{i}\right)-\mu_{Z}\left(\omega^{j}_{1}\right)\right| \leq \frac{\displaystyle\sum_{l = 0}^{j-1}\rho_{k-i+l,X}}{(1-\xi)^{j}},
$$
where the sequence  $\left(\rho_{l,X}\right)_{l \in \mathbb{N} }$ satisfies (\ref{galeo2008}).
\end{proof}

The proof of next result is a consequence of Proposition $4$ of Dedecker and Doukhan (2003).

\begin{lemma}\label{enlema10}
For any finite sequence $\omega$ and any $t>0$, we have
$$
\mathbb{P}\left(\left|N_{n}(\omega)-(n-l(\omega)+1)\mu_{Z}(\omega)\right|>t\right)\leq e^{\frac{1}{e}}\exp\left[-\frac{-t^{2}(1-\varepsilon)^{l(\omega)}}{4e[n-l(\omega)+1]l(\omega)\left(1+\frac{\beta_{X}}{\alpha_{X}}\right)}\right].
$$
Also, for any $a \in {\cal{A}}$ and any $n>\frac{|{\cal{A}}|+1}{tq(\omega)}+l(\omega)$, we  have
$$
\mathbb{P}\left(\left|{\hat{p}_(a|\omega)_n}-p_{Z}(a|\omega)\right|>t\right)\leq 3e^{\frac{1}{e}}\exp\left\{-(n-l(\omega))\frac{\left[t-\frac{3}{(n-l(\omega))\mu_{Z}(\omega)}\right]^{2}{\mu_{Z}(\omega)}^{2}(1-\varepsilon)^{l(\omega a)}}{64el(\omega a)\left(1+\frac{\beta_{X}}{\alpha_{X}}\right)}\right\}\cdot
$$
\end{lemma}
\begin{proof}
For the model (\ref{defproczt}) we have, for any finite sequence $\omega_{1}^{j} \in {\cal{A}}_{1}^{j}$, that
\begin{eqnarray*}
N_{n}(\omega_{1}^{j})&=& \sum_{t = 0}^{n-j} \displaystyle\prod_{i\in{\cal{I}}^{*}} 
\left[{\textbf{1}}_{\left\{X_{t+i}= \omega_{i}\right\}} {\textbf{1}}_{\left\{\xi_{t+i}= 1\right\}}\right] \displaystyle\prod_{s\in{\cal{S}}^{*}} \left[{\textbf{1}}_{\left\{X_{t+s}= \omega_{s}\right\}} + {\textbf{1}}_{\left\{X_{t+s}= 1\right\}} {\textbf{1}}_{\left\{\xi_{t+s}= 0\right\}} \right],
\end{eqnarray*}
where ${\cal{I}}^{*}=\left\{k \ : \ 1\leq k \leq j, \ \omega _{k}\neq 0\right\}$ and ${\cal{S}}^{*}=\left\{l\ : \ 1\leq l \leq j, \ \omega _{l} = 0\right\}$. 
Define the process {\bf U} by
$$
U_{t}=\displaystyle\prod_{i\in{\cal{I}}^{*}}\left[{\textbf{1}}_{\left\{X_{t+i}= \omega_{i}\right\}} {\textbf{1}}_{\left\{\xi_{t+i}= 1\right\}}\right] \displaystyle\prod_{s\in{\cal{S}}^{*}} \left[{\textbf{1}}_{\left\{X_{t+s}= \omega_{s}\right\}} + {\textbf{1}}_{\left\{X_{t+s}= 1\right\}} {\textbf{1}}_{\left\{\xi_{s+t}= 0\right\}} \right] - \mu_{Z}(\omega_{1}^{j}).
$$
Denote by ${\cal{M}}_{i}$ the $\sigma$-algebra generated by $U_{0},...,U_{i}$. Note that
$\mathbb{E}\left(U_{t}\right)=0$ and $\left\|U_{t}\right\|_{\frac{r}{2}}\leq 1$. Now, by applying Proposition $4$ of Dedecker and Doukhan (2003), we obtain
\begin{equation}\label{despro9}
\left\|N_{n}(\omega_{1}^{j})-(n-j+1)\mu_{Z}(\omega_{1}^{j})\right\|_{r}\leq \left(2r \sum_{t = 0}^{n-j}\underset{t\leq l \leq n-j}{\max} \left\|U_{t} \sum_{k = t}^{l}\mathbb{E}(U_{k}|{\cal{M}}_{t})\right\|_{r/2}\right)^{1/2},
\end{equation}
It follows by definition of $\left\|.\right\|_{\infty}$ that
\begin{eqnarray}
\left\|\mathbb{E}(U_{k}|{\cal{M}}_{t})\right\|_{\infty}
=\sup_{x_{1}^{t+j}, \theta_{1}^{t+j} }\left|\mathbb{P}(Z_{k+j}^{k+1}=\omega_{j}^{1}|X_{1}^{t+j}=x_{1}^{t+j}, \xi_{1}^{t+j}=\theta_{1}^{t+j}) - \mu_{Z}(\omega_{1}^{j})\right|,\label{nanananan}
\end{eqnarray}
where $x_{1}^{t+j}, \theta_{1}^{t+j} \in {\cal{A}}_{1}^{t+j}$. Plugging (\ref{nanananan}) into (\ref{despro9}) and applying Lemma \ref{lema88888}, we have
\begin{equation}\label{despro9123134}
\left\|N_{n}(\omega_{1}^{j})-(n-j+1)p_{Z}(\omega_{1}^{j})\right\|_{r} \leq \left[\frac{4r}{(1-\varepsilon)^{j}} (n-j+1)j  \left( 1 + \frac{\beta_{X}}{\alpha_{X}}\right)\right]^{1/2}.
\end{equation}
Now, let 
$$
B=\frac{4}{(1-\varepsilon)^{j}} (n-j+1)j  \left( 1 + \frac{\beta_{X}}{\alpha_{X}}\right).
$$
Then, as in  Dedecker e Prieur (2005) we obtain for any $t>0$ that
\begin{eqnarray}
\mathbb{P}\left( \left|N_{n}(\omega_{1}^{j})-(n-j+1)\mu_{Z}(\omega_{1}^{j})\right| > t \right) \leq \min\left(1,\left[\frac{rB}{t^{2}}\right]^{\frac{r}{2}}\right).\label{despro992}
\end{eqnarray}
Now, following Galves and Leonard (2008), one can infer that
\begin{equation}\label{despro9924447}
\min\left(1,\left[\frac{rB}{t^{2}}\right]^{\frac{r}{2}}\right)\leq \exp\left\{-\frac{t^{2}}{eB}+e^{-1}\right\}.
\end{equation}
Thus, by plugging (\ref{despro9924447}) into (\ref{despro992}) the first assertion of the lemma follows.
In order to prove the second one, we note that
\begin{equation}\label{ddespro91}
\left|p_{Z}(a|\omega)-\frac{(n-l(\omega))\mu_{Z}(\omega a)+1}{(n-l(\omega))\mu_{Z}(\omega)+|{\cal{A}}|}\right|\leq \frac{|{\cal{A}}|+1}{(n-l(\omega))\mu_{Z}(\omega)}\cdot
\end{equation}
Thus, we can write
\begin{eqnarray*}
\lefteqn{\mathbb{P}\left(\left|p_{Z}(a|\omega)-\hat{p}_{Z}(a|\omega)_n\right| > t\right)} \\
&\leq&\mathbb{P}\left(\left|\hat{p}_{Z}(a|\omega)_n-\frac{(n-l(\omega))\mu_{Z}(\omega a)+1}{(n-l(\omega))\mu_{Z}(\omega)+|{\cal{A}}|}\right| >t-\frac{|{\cal{A}}|+1}{(n-l(\omega))\mu_{Z}(\omega)}\right),
\end{eqnarray*}
for any $n\geq\frac{|{\cal{A}}|+1}{tq(\omega)}+l(\omega)$. Making $t'=t-\frac{|{\cal{A}}|+1}{(n-l(\omega))\mu_{Z}(\omega)}$, we can see that 
\begin{eqnarray*}
\lefteqn{\mathbb{P}\left(\left|\hat{p}_{Z}(a|\omega)_n-\frac{(n-l(\omega))\mu_{Z}(\omega a)+1}{(n-l(\omega))\mu_{Z}(\omega)+|{\cal{A}}|}\right| >t'\right)}\\
&\leq&\mathbb{P}\left(\left|N_{n}(\omega a)-(n-l(\omega))\mu_{Z}(\omega a)\right|>\frac{t'}{2}\left[(n-l(\omega))\mu_{Z}(\omega)+|{\cal{A}}|\right]\right)\\
&+& \sum_{b \in {\cal{A}}}\mathbb{P}\left(\left|N_{n}(\omega b)-(n-l(\omega))\mu_{Z}(\omega b)\right|>\frac{t'}{2|{\cal{A}}|}\left[(n-l(\omega))\mu_{Z}(\omega)+|{\cal{A}}|\right]\right)\cdot
\end{eqnarray*}
Therefore, by applying the first assertion of this lemma in the second expression of last inequality we conclude the proof of this lemma.
\end{proof}

\begin{lemma}\label{enlema11}
For any $\delta >2\left[1+\frac{4\beta_{X}}{\min(\alpha_{X} \beta_{X}^{*},1)}\right]\varepsilon$, any $\omega \in {\cal{T}}_{X}$, $u\omega \in \hat{{\cal{T}}}_{n}^{\delta, d}$ and 
$$
n>\frac{6}{\left\{\delta-2\varepsilon\left[1+\frac{4\beta_{X}}{\min{(1,(1+\varepsilon)\alpha_{X}\beta_{X}^{*})}}\right]\right\}\alpha_{X}^{d}(1-\varepsilon)^{d}}+d
$$
we have 
$$
\mathbb{P}(\Delta_{n}(u\omega)>\delta)\leq 12e^{\frac{1}{e}}\exp{\left\{-(n-d)\frac{\left[\frac{\delta}{2}-\bar{k}\right]^{2}{\alpha_{X}^{2d}}(1-\varepsilon)^{3d+1}}{ 64e(d+1)  \left( 1 + \frac{\beta_{X}}{\alpha_{X}}\right)}\right\}},
$$
where 
\begin{equation}\label{kbarra}
\bar{k}=\varepsilon\left[1+\frac{4\beta}{\min{(1,(1+\varepsilon)\alpha_{X}\beta_{X}^{*})}}\right]+\frac{3}{(n-d)\alpha_{X}^{d}(1-\varepsilon)^{d}}\cdot
\end{equation}
\end{lemma}
\begin{proof}
By the suffix property (item $(2)$ of the Definition \ref{sufixo111}) if $\omega \in {\cal{T}}_{X}$ then 
\begin{equation}
\label{eee1}
p_{X}(a|u\omega) = p_{X}(a|\mbox{suf}(u\omega)), 
\end{equation}
for any finite sequence $u$ and any symbol $a \in {\cal{A}}$. Thus,
from (\ref{eee1}), by triangle inequality, we have
\begin{eqnarray*}
\lefteqn{\left|\hat{p}_Z(a|u\omega)_n-\hat{p}_Z(a|\mbox{suf}(u\omega)_n)\right|\leq\left|\hat{p}_Z(a|u\omega)_n-p_{Z}(a|u\omega)\right| + \left|p_{X}(a|u\omega)-p_{Z}(a|u\omega)\right|} \nonumber\\
&+&\left|p_{X}(a|\mbox{suf}(u\omega))-p_{Z}(a|\mbox{suf}(u\omega))\right|
+\left|p_{Z}(a|\mbox{suf}(u\omega))-\hat{p}_{Z}(a|\mbox{suf}(u\omega))_n\right|\cdot
\end{eqnarray*}
Then, considering Theorem \ref{enteor1}, the following inequality holds
\begin{eqnarray*}
\lefteqn{\mathbb{P}\left(\underset{a\in {\cal{A}}}{\max }\left|\hat{p}_Z(a|u\omega)_n-\hat{p}_Z(a|\mbox{suf}(u\omega))_n\right|>\delta\right)}\\
&\leq& \sum_{a\in {\cal{A}}}\left[\mathbb{P}\left(\left|\hat{p}_Z(a|u\omega)_n-p_{Z}(a|\mbox{suf}(u\omega))\right|>\frac{\delta}{2}-\varepsilon\left[1+\frac{4\beta_{X}}{\min{(1,(1+\varepsilon)\alpha_{X}\beta_{X}^{*})}}\right]\right)\right.\nonumber\\
&+&\left.\mathbb{P}\left(\left|p_{Z}(a|\mbox{suf}(u\omega))-\hat{p}_{Z}(a|\mbox{suf}(u\omega))_n\right|>\frac{\delta}{2}-\varepsilon\left[1+\frac{4\beta_{X}}{\min{(1,(1+\varepsilon)\alpha_{X}\beta_{X}^{*})}}\right]\right)\right]\cdot\nonumber\\
\end{eqnarray*}
Therefore, by using Lemmas \ref{enlema5} and \ref{enlema10} to bound from above the second expression of last inequality, the lemma follows.
\end{proof}

\begin{lemma}\label{enlema12345}
There exists $d$ such that, for any 
$$\delta <D_{d} - 2\varepsilon\left[1 + \frac{4\beta_{X}}{\min{(1,(1+\varepsilon)\alpha_{X}\beta_{X}^{*})}}\right],
$$ 
any $\omega\in\hat{{\cal{T}}}_{n}^{\delta,d}$, with $l(\omega)<K,$ $\omega \notin {\cal{T}}_{X}$ and any
$$
n>\frac{6}{\left\{D_{d} - \delta -2\varepsilon\left[1 + \frac{4\beta_{X}}{\min{(1,(1+\varepsilon)\alpha_{X}\beta_{X}^{*})}}\right]\right\}\alpha_{X}^{d}(1-\varepsilon)^{d}}+d,
$$
we have
$$
\mathbb{P}\left(\bigcap_{u\omega \in {\cal{T}}_{X}\left.\right|_{d}}\left\{\Delta_{n}(u\omega)\leq \delta\right\}\right)\leq 6e^{\frac{1}{e}}\exp\left[-(n-d)\frac{\left[\frac{D_{d} - \delta}{2}-\bar{k}\right]^{2}{\alpha_{X}^{2d}(1-\varepsilon)^{3d+1}}}{64e(d+1)\left(1+\frac{\beta_{X}}{\alpha_{X}}\right)}\right],
$$
where $\bar{k}$ is as in (\ref{kbarra}). 
\end{lemma}
\begin{proof}
Similarly to Collet,  Galves  and Leonardi (2008), we take
$$
d=\underset{u\notin {\cal{T}}_{X}, \ l(u)<K}{\max}\min{\left\{k:\mbox{ existe } \omega \in C_{k} \ \mbox{com}  \ \omega\succ u\right\}}.
$$
Then, from the definitions of $C_{d}$ and $d$, there exists $\bar{u\omega}\in {\cal{T}}_{X}\left.\right|_{d}$, such that 
$
p_{X}(a|\bar{u\omega})\neq p_{X}(a|\mbox{suf}(\bar{u\omega}))
$ for some $a \in {\cal{A}}$. Now, we can see that
\begin{eqnarray*}
\lefteqn{\left|\hat{p}_Z(a|\bar{u\omega})_n-\hat{p}_Z(a|\mbox{suf}(\bar{u\omega}))_n\right| \geq \left|p_{X}(a|\bar{u\omega})-p_{X}(a|\mbox{suf}(\bar{u\omega}))\right|}\\
&-&\left|\hat{p}_Z(a|\mbox{suf}(\bar{u\omega}))_n-p_{Z}(a|\mbox{suf}(\bar{u\omega}))\right|-\left|p_{Z}(a|\mbox{suf}(\bar{u\omega}))-p_{X}(a|\mbox{suf}(\bar{u\omega}))\right|\\
&-&\left|p_{Z}(a|\bar{u\omega})-p_{X}(a|\bar{u\omega})\right|-\left|\hat{p}_Z(a|\bar{u\omega})_n-p_{Z}(a|\bar{u\omega})\right|
\end{eqnarray*}
for all $a \in {\cal{A}}$. Thus, if $\bar{u\omega}\in C_{d}$, then
\begin{equation*}
\underset{a\in {\cal{A}}}{\max}\left|p_{X}(a|\bar{u\omega})-p_{X}(a|\mbox{suf}(\bar{u\omega}))\right|\geq D_{d}.
\end{equation*}
From Theorem \ref{enteor1}, we can write
\begin{eqnarray*}
\Delta_{n}(\bar{u\omega})&\geq& D_{d} - 2\varepsilon\left[1 + \frac{4\beta_{X}}{\min{(1,(1+\varepsilon)\alpha_{X}\beta_{X}^{*})}}\right]\\
&-&\underset{a\in {\cal{A}}}{\max}\left|\hat{p}_Z(a|\bar{u\omega})_n-p_{Z}(a|\bar{u\omega})\right|-\underset{a\in {\cal{A}}}{\max}\left|\hat{p}_Z(a|\mbox{suf}(\bar{u\omega}))_n-p_{Z}(a|\mbox{suf}(\bar{u\omega}))\right|\nonumber\cdot
\end{eqnarray*}
From the last inequality and Lemma \ref{enlema10} we have
$$
\mathbb{P}\left(\bigcap_{u\omega \in {\cal{T}}_{X}\left.\right|_{d}}\left\{\Delta_{n}(u\omega)\leq \delta\right\}\right)\leq  6e^{\frac{1}{e}}\exp\left[-(n-d)\frac{c^{2}{\alpha_{X}^{2d}(1-\varepsilon)^{3d+1}}}{64e(d+1)\left(1+\frac{\beta_{X}}{\alpha_{X}}\right)}\right],
$$
where
$
c=\frac{D_{d} - \delta}{2} -\varepsilon\left[1 + \frac{4\beta_{X}}{\min{(1,(1+\varepsilon)\alpha_{X}\beta_{X}^{*})}}\right]-\frac{3}{(n-d)\alpha_{X}^{d}(1-\varepsilon)^{d}}\cdot
$
This concludes the proof of this lemma.
\end{proof}

\begin{pft1}
{\rm 
Following B\"uhlmann and Wyner (1999), we define
$$
O_{n,\delta}^{K,d}=\bigcup_{\underset{l(\omega)<K}{\omega \in {\cal{T}}_{X}}}\bigcup_{u\omega \in \hat{{\cal{T}}}_{n}^{\delta, d}}\left\{\Delta_{n}(u\omega)>\delta\right\},
$$
and
$$
U_{n,\delta}^{K,d}=\bigcup_{\underset{l(\omega)<K}{\omega \in \hat{{\cal{T}}}_{n}^{\delta,d}}}\bigcap_{u\omega \in {\cal{T}}_{X}\left.\right|_{d}}\left\{\Delta_{n}(u\omega)\leq\delta\right\}.
$$
Then, if $d<n$, one can see that
$$
\left\{\hat{{\cal{T}}}_{n}^{\delta,d}\left.\right|_{K}\neq {\cal{T}}_{X}\left.\right|_{K}\right\}\subseteq O_{n,\delta}^{K,d}\cup U_{n,\delta}^{K,d}.
$$
Thus,
$$
\mathbb{P}\left(\hat{{\cal{T}}}_{n}^{\delta,d}\left.\right|_{K}\neq {\cal{T}}_{X}\left.\right|_{K}\right)\leq\displaystyle\sum_{\underset{l(\omega)<K}{\omega \in {\cal{T}}_{X}}}\displaystyle\sum_{u\omega \in \hat{{\cal{T}}}_{n}^{\delta, d}}\mathbb{P}\left(\Delta_{n}(u\omega)>\delta\right)+\displaystyle\sum_{\underset{l(\omega)<K}{\omega \in \hat{{\cal{T}}}_{n}^{\delta,d}}}\mathbb{P}\left(\bigcap_{u\omega \in {\cal{T}}_{X}\left.\right|_{d}}\left\{\Delta_{n}(u\omega)\leq\delta\right\}\right).
$$
Therefore, from Definition \ref{estarvcont}, Lemmas \ref{enlema11} and \ref{enlema12345}, we have
$$
\mathbb{P}\left(\hat{{\cal{T}}}_{n}^{\delta,d}\left.\right|_{K}\neq {\cal{T}}_{X}\left.\right|_{K}\right)\leq 2^{d}12e^{\frac{1}{e}}\exp\left[-(n-d)\frac{\left[\min\left(D_{d} - \delta,\delta\right)-2\bar{k}\right]^{2}{\alpha_{X}^{2d}(1-\varepsilon)^{3d+1}}}{256e(d+1)\left(1+\frac{\beta_{X}}{\alpha_{X}}\right)}\right],
$$
where $\bar{k}$ is as in (\ref{kbarra}).
It completes the proof of Theorem \ref{enteorema2}.}
\end{pft1}

\begin{pfc}

By the Theorem \ref{enteorema2}, we have
\begin{eqnarray*}
\displaystyle\sum_{n}\mathbb{P}\left(\hat{{\cal{T}}}_{n}^{\delta,d}\left.\right|_{K}\neq {\cal{T}}_{X}\left.\right|_{K}\right)\leq \displaystyle\sum_{n}c_{2}e^{\left[-(n-d)c_{3}\right]}<\infty,
\end{eqnarray*}
for suitable choices of $d$ and $\delta$. Thus, by the Borel-Cantelli Lemma, we obtain
$$
\mathbb{P}\left(\left[\hat{{\cal{T}}}_{n}^{\delta,d}\left.\right|_{K}\neq {\cal{T}}_{X}\left.\right|_{K} i.o.\right]\right)=0.
$$
\end{pfc}


\subsection{Proof of Theorem \ref{enteor31}}

The proof of Theorem \ref{enteor31} is based on three preparatory lemmas.

\begin{lemma}\label{3lema4}
  For any $\varepsilon \in (0,1)$, any $k>j\geq 0$, any
  $\omega_{-\infty}^{-a} \in {\cal A}_{-\infty}^{-1}$, $\omega_0 \in
  {\cal A}$ and any $a,b \in {\cal{A}}$, we have
$$
\left|\mathbb{P}\left(X_{0}=\omega_{0}|X_{-j}^{-1}=\omega_{-j}^{-1}, X_{-j-1}=a, Z_{-j-1}=b, Z_{-k}^{-j-2}=\omega_{-k}^{-j-2}\right)-p_{X}(\omega_{0}|\omega_{-\infty}^{-1}) \right|\leq\beta_{j,X}.
$$
\end{lemma}
\begin{proof}
For any $j\geq 0$, considering the independence of the processes ${\boldsymbol \xi}$, {\bf
  Y} and {\bf X}, we have
\begin{eqnarray*}
\lefteqn{\mathbb{P}\left(X_{0}=\omega_{0}|X_{-j}^{-1}=\omega_{-j}^{-1}, X_{-j-1}=a, Z_{-j-1}=b, Z_{-k}^{-j-2}=\omega_{-k}^{-j-2}\right)}
\\&=&\frac{\displaystyle\sum_{u_{-k}^{-j-2}}p_{X}\left(u_{-k}^{-j-2}a\omega_{-j}^{-1}\omega_{0}\right)\mathbb{P}\left(Z_{-k}^{-j-1}=\omega_{-k}^{-j-2}b|X_{-k}^{-j-1}=u_{-k}^{-j-2}a\right)}{\displaystyle\sum_{u_{-k}^{-j-2}}p_{X}\left(u_{-k}^{-j-2}a\omega_{-j}^{-1}\right)\mathbb{P}\left(Z_{-k}^{-j-1}=\omega_{-k}^{-j-2}b|X_{-k}^{-j-1}=u_{-k}^{-j-2}a\right)}\cdot
\end{eqnarray*}
Since
\begin{equation*}
p_{X}\left(\omega_{0}|\omega_{-\infty}^{-1}\right)-\beta_{j,X}\leq p_{X}\left(\omega_{0}|u_{-k}^{-j-2}a\omega_{-j}^{-1}\right)\leq p_{X}\left(\omega_{0}|\omega_{-\infty}^{-1}\right)+\beta_{j,X},
\end{equation*}
the assertion of this lemma follows directly.
\end{proof}



\begin{lemma}\label{3enlema5}
For any $\varepsilon \in (0,1)$, any $k\geq 0$ and any $\omega_{-k}^{0}$, we have
\begin{equation}\label{lkjhhgz00}
p_{Z}\left(\omega_{0}|\omega_{-k}^{-1}\right)\geq \alpha_{\min},
\end{equation}
\begin{equation}\label{lkjhhgz}
\mathbb{P}\left(X_{0}=\omega_{0}|Z_{-k}^{-1}=\omega_{-k}^{-1}\right)\geq\alpha_{X},
\end{equation}
\begin{equation}\label{ymod22}
\mathbb{P}\left(Y_{0}=\omega_{0}|Z_{-k}^{-1}=\omega_{-k}^{-1}\right)\geq\alpha_{Y},
\end{equation}
where $\alpha_{\min}=\min{\{\alpha_{X},\alpha_{Y}\}}$. Moreover, for any $0\leq j\leq k$, we have
\begin{equation}\label{xmod2}
\mathbb{P}\left(X_{-j-1}=\omega_{-j-1}|X_{-j}^{-1}=\omega_{-j}^{-1}, Z_{-k}^{-j-2}=\omega_{-k}^{-j-2}\right)\geq \alpha_{X}\beta_{X}^{*},
\end{equation}
where $\beta_{X}^{*}=\prod_{k=0}^{+\infty}(1-\beta_{k, X})>0$.
\end{lemma}

\begin{proof}
It follows from (\ref{lulu1}) that
\begin{equation}\label{luview44}
p_{Z}\left(\omega_{0}|\omega_{-k}^{-1}\right)=(1-\varepsilon)\mathbb{P}\left(X_{0}=\omega_{0}|Z_{-k}^{-1}=\omega_{-k}^{-1}\right) + \varepsilon\mathbb{P}\left(Y_{0}=\omega_{0}|Z_{-k}^{-1}=\omega_{-k}^{-1}\right).
\end{equation}
Thus the assertion (\ref{lkjhhgz00}) follows from (\ref{lkjhhgz}), (\ref{ymod22}) and (\ref{luview44}). Now, considering the independence of the
processes ${\boldsymbol \xi}$, {\bf X} and {\bf Y}, we have
\begin{eqnarray}
\lefteqn{\mathbb{P}\left(X_{0}=\omega_{0}|Z_{-k}^{-1}=\omega_{-k}^{-1}\right)}\label{nanviu}\\
&=&\lim_{l\rightarrow +\infty}\frac{\displaystyle\sum_{v_{-k}^{-1}}\displaystyle\sum_{u_{-l}^{-1}}p_{X}\left(\omega_{0}|\omega_{-\infty}^{-l-1}u_{-l}^{-1}\right)\mathbb{P}\left(X_{-l}^{-1}=u_{-l}^{-1}|X_{-\infty}^{-l-1}=\omega_{{-\infty}}^{-l-1}\right)p_{Y}(v_{-k}^{-1})p_{\xi}}{\displaystyle\sum_{v_{-k}^{-1}}\displaystyle\sum_{u_{-l}^{-1}}\mathbb{P}\left(X_{-l}^{-1}=u_{-l}^{-1}|X_{-\infty}^{-l-1}=\omega_{{-\infty}}^{l-1}\right)p_{Y}(v_{-k}^{-1})p_{\xi}},\nonumber
\end{eqnarray}
where
$$
p_{\xi}=\mathbb{P}\left(\displaystyle\bigcap_{i: \ -k \leq -i \leq -1,  \ u_{-i}\neq v_{-i}, \ u_{-i}= \omega_{-i}}\xi_{-i}=1,\displaystyle\bigcap_{t: \ -k \leq -t \leq -1,  \ \ u_{-t}\neq v_{-t}, \ u_{-t}\neq \omega_{-t}}\xi_{-t}=0\right),
$$
with the sequences $u_{-l}^{-1}\in {\cal{A}}_{-l}^{-1}$, $v_{-k}^{-1}\in {\cal{A}}_{-k}^{-1}$, satisfying 
$$
\mathbb{P}\left(X_{-l}^{-1}=u_{-l}^{-1}, Y_{-k}^{-1}=v_{-k}^{-1}, Z_{-k}^{-1}=\omega_{-k}^{-1}\right)\neq 0.
$$
From the non-nullness hypothesis of the process  {\bf X} and (\ref{nanviu}) we can write (\ref{lkjhhgz}). Analo-\newline gously, we can prove (\ref{ymod22}). Now, in order to prove (\ref{xmod2}), we note that
\begin{eqnarray}\label{modfff}
\lefteqn{\mathbb{P}\left(X_{-j-1}=\omega_{-j-1}|X_{-j}^{-1}=\omega_{-j}^{-1}, Z_{-k}^{-j-2}=\omega_{-k}^{-j-2}\right)}\\
&=&\frac{\displaystyle\sum\mathbb{P}\left(Z_{-k}^{-j-2}=\omega_{-k}^{-j-2}|X_{-k}^{-j-2}=x_{-k}^{-j-2},Y_{-k}^{-j-2}=y_{-k}^{-j-2}\right)p_{X}\left(x_{-k}^{-j-2}\omega_{-j-1}^{-1}\right)p_{Y}(y_{-k}^{-j-2})}{\displaystyle\sum\mathbb{P}\left(Z_{-k}^{-j-2}=\omega_{-k}^{-j-2}|X_{-k}^{-j-2}=x_{-k}^{-j-2},Y_{-k}^{-j-2}=y_{-k}^{-j-2}\right)p_{X}\left(x_{-k}^{-j-2}\omega_{-j}^{-1}\right)p_{Y}(y_{-k}^{-j-2})}\cdot\nonumber
\end{eqnarray}
The last two summations are over the set 
\begin{equation*}
{\cal{S}^{*}}=\left\{x_{-k}^{-j-2}, \ y_{-k}^{-j-2} \in {\cal{A}}_{-k}^{-j-2}\left|\right.\left\{X_{-k}^{-j-2}=x_{-k}^{-j-2}, Y_{-k}^{-j-2}=y_{-k}^{-j-2}, Z_{-k}^{-j-2}=\omega_{-k}^{-j-2}\right\}\neq\emptyset\right\}.
\end{equation*}
Therefore, by plugging (\ref{e.eeee}) into (\ref{modfff}), the assertion (\ref{xmod2}) follows. 
\end{proof}


\begin{lemma}\label{3lema6n}
For any $\varepsilon \in (0,1)$, any $k>j\geq 0$ and any $\omega_{-k}^{0}$, we have
$$
\mathbb{P}\left(X_{-j-1}={\omega }_{-j-1}'|X_{-j}^{-1}=\omega_{-j}^{-1}, Z_{-k}^{-j-1}=\omega_{-k}^{-j-1}\right)\leq \frac{\varepsilon}{\alpha\beta^{*}_{\min}},
$$ 
where ${\omega }_{-j-1}'\neq \omega _{-j-1}$ and $\alpha\beta^{*}_{\min}=\min\{\alpha_{X}\beta_{X}^{*},\alpha_{Y}\}$.
\end{lemma}

\begin{proof}
Analogously to the proof of Lemma \ref{lema6}, we can show that
\begin{eqnarray}\label{newkkk}
\lefteqn{\mathbb{P}\left(X_{-j-1}={{\omega }_{-j-1}'}|X_{-j}^{-1}=\omega_{-j}^{-1}, Z_{-k}^{-j-1}=\omega_{-k}^{-j-1}\right)}\nonumber \\
&\leq&\frac{\varepsilon}{\mathbb{P}\left(Z_{-j-1}=\omega_{-j-1}|X_{-j}^{-1}=\omega_{-j}^{-1}, Z_{-k}^{-j-2}=\omega_{-k}^{-j-2}\right)}\cdot\label{mod45rr}
\end{eqnarray}
Now, from the independence between the processes {\bf X} and {\bf Y}, the inequality (\ref{ymod22}), it follows that
\begin{eqnarray}
\mathbb{P}\left(Y_{-j-1}=\omega_{-j-1}|X_{-j}^{-1}=\omega_{-j}^{-1}, Z_{-k}^{-j-2}=\omega_{-k}^{-j-2}\right)
&\geq& \alpha_{Y}.\label{modfgh}
\end{eqnarray}
From (\ref{modfgh}), (\ref{xmod2}) and (\ref{newkkk}), the lemma is proved.
\end{proof}


\begin{pft3}
{\rm 
First, for any $a \in {\cal{A}}$ and any $\omega_{-k}^{-1} \in {\cal{A}}_{-k}^{-1}$, we have 
\begin{equation}\label{con2222}
\left| p_{Z}\left( a | \omega_{-k}^{-1} \right) - \mathbb{P}\left( X_{0}=a | Z_{-k}^{-1}=\omega_{-k}^{-1}\right) \right| \leq 2\varepsilon\cdot
\end{equation}
Proceeding analogously to the proof of Theorem \ref{enteor1}, taking account Lemmas \ref{3lema4} and \ref{3lema6n},  one can show that
\begin{equation*}
\mathbb{P}(X_{0}=a|Z_{-k}^{-1}=\omega_{-k}^{-1})-\mathbb{P}(X_{0}=a|X_{-k}^{-1}=\omega_{-k}^{-1})\leq \displaystyle\sum_{j=0}^{k-1}\left[(N-1)\frac{2\varepsilon\beta_{j,X}}{\alpha\beta^{*}_{\min}} + (N-1)2\varepsilon \beta_{j,X}\right].
\end{equation*}
Therefore, considering the last inequality and  (\ref{con2222}), we have
$$
\left|\mathbb{P}(X_{0}=a|Z_{-k}^{-1}=\omega_{-k}^{-1})-\mathbb{P}(X_{0}=a|X_{-k}^{-1}=\omega_{-k}^{-1})\right|\leq 2\varepsilon+(N-1)\left[\frac{2\varepsilon\beta_{X}}{\alpha\beta^{*}_{\min}}+2\varepsilon \beta_{X}\right]
$$
and the theorem is proved.}
\end{pft3}


\subsection{Proof of Theorem \ref{3enteo2}}


\begin{lemma}\label{33gale3} For any $i\geq 1$, any $k \geq i$, any
  $j\geq 1$ and any finite sequence $\omega_{1}^{j} \in {\cal A}_{1}^{j}$, the following
  inequality holds
$$
\sup_{x^{i}_{1}, \ y^{i}_{1} \in {\cal{A}}^{i}; \ \theta_{1}^{i} \in {\{0,1\}}^{i}}\left|\mathbb{P}\left(Z_{k}^{k+j-1}=\omega^{j}_{1}\left.\right|X_{1}^{i}=x_{1}^{i}, Y_{1}^{i}=y_{1}^{i}, \xi_{1}^{i}=\theta_{1}^{i}\right)-p_{Z}\left(\omega^{j}_{1}\right)\right| \leq \frac{2\rho_{ijk, \max}}{(1-\varepsilon)^{j}},
$$ 
where
$$
\rho_{ijk, \max}=\max{\left\{\displaystyle\sum_{l = 0}^{j-1}\rho_{k-i+l, X},\displaystyle\sum_{s = 0}^{j-1}\rho_{k-i+s, Y}\right\}}.
$$
\end{lemma}


\begin{proof}
From the independence of processes {\bf X}, {\bf Y} and
${\boldsymbol \xi}$, for any $x^{i}_{1}, \ y^{i}_{1} \in {\cal{A}}^{i}$ and any $\theta_{1}^{i} \in {\{0,1\}}^{i}$, we have 
\begin{eqnarray*}
\lefteqn{\left|\mathbb{P}\left({Z_{k}}^{k+j-1}=\omega^{j}_{1}\left.\right|X_{1}^{i}=x_{1}^{i}, Y_{1}^{i}=y_{1}^{i}, \xi_{1}^{i}=\theta_{1}^{i}\right)-p_{Z}\left(\omega^{j}_{1}\right)\right|}\\
&=&\left|\displaystyle\sum\mathbb{P}\left(X_{k}^{k+j-1}= x_{k}^{k+j-1}, Y_{k}^{k+j-1}= y_{k}^{k+j-1}, Z_{k}^{k+j-1}=\omega_{1}^{j}|X_{1}^{i} = x_{1}^{i},Y_{1}^{i} = y_{1}^{i}, \xi_{1}^{i}=\theta_{1}^{i}\right)\right.\\
& & - \left.p_{Z}\left(\omega^{j}_{1}\right)\right|\\
&&\quad \times\left|\mathbb{P}\left(X_{k}^{k+j-1}=x_{k}^{k+j-1},Y_{k}^{k+j-1}=y_{k}^{k+j-1}|X_{1}^{i}=x_{1}^{i}, Y_{1}^{i}=y_{1}^{i}\right) -p_{X}(x_{k}^{k+j-1})p_{Y}(y_{k}^{k+j-1})\right|\\
&\leq&\displaystyle\sum\mathbb{P}\left(Z_{k}^{k+j-1}=\omega^{j}_{1}|X_{k}^{k+j-1}=x_{k}^{k+j-1},Y_{k}^{k+j-1}=y_{k}^{k+j-1}\right)\\
&&\quad \times\left[\left|\mathbb{P}\left(X_{k}^{k+j-1}=x_{k}^{k+j-1},Y_{k}^{k+j-1}=y_{k}^{k+j-1}|X_{1}^{i}=x_{1}^{i}, Y_{1}^{i}=y_{1}^{i}\right)\right.\right.\\
& & - \left.\left.p_{Y}(y_{k}^{k+j-1}|y_{1}^{i})p_{X}(x_{k}^{k+j-1})\right|\right.\\
& & + \left.\left|p_{Y}(y_{k}^{k+j-1}|y_{1}^{i})p_{X}(x_{k}^{k+j-1})-p_{X}(x_{k}^{k+j-1})p_{Y}(y_{k}^{k+j-1})\right|\right]\\
&\leq&\frac{2\rho_{ijk, \max}}{(1-\xi)^{j}},
\end{eqnarray*} 
where the summations are over ${x_{k}^{k+j-1}, \ y_{k}^{k+j-1}\in {\cal{A}}_{1}^{j}}$. Since the {\bf X} and {\bf Y} satisfy the assumptions of non-nullness
and  summability of the continuity rate the last inequality is a consequence the Lemma $3.4$ of Galves and Leonardi (2008).  
\end{proof}

We observe that from this point on the technique used in the proof of the Theorem  \ref{3enteo2} will be essentially the same  employed in the proof of Theorem \ref{enteorema2}.



\begin{lemma}\label{33lem10}
For any finite sequence $\omega$ and any $t > 0$, we have
\begin{eqnarray}\label{33lem10nancy}
\mathbb{P}\left(\left|N_{n}(\omega)-(n-l(\omega)+1)p_{Z}(\omega)\right|>t\right)\leq e^{\frac{1}{e}}\exp\left[-\frac{-t^{2}(1-\varepsilon)^{l(\omega)}}{4e[n-l(\omega)+1]l(\omega){\beta}_{\alpha, \max}}\right],
\end{eqnarray}
where ${\beta}_{\alpha, \max}=\max{\left\{\left( 1 + \frac{\beta_{X}}{\alpha_{X}}\right),\left( 1 + \frac{\beta_{Y}}{\alpha_{Y}}\right)\right\}}$. Moreover, for any $a \in {\cal{A}}$ and any $n>\frac{N+1}{tq(\omega)}+l(\omega)$, we have
\begin{eqnarray*}
\lefteqn{\mathbb{P}\left(\left|{\hat{p}_{{Z}_{n}}(a|\omega)}-p_{Z}(a|\omega)\right|>t\right)}\\
&\leq&(N+1)e^{\frac{1}{e}}\exp{\left\{-\left[t-\frac{N+1}{(n-l(\omega))p_{Z}(\omega)}\right]^{2}\frac{{\left[p_{Z}(\omega)\right]}^{2}(n-l(\omega))(1-\varepsilon)^{l(\omega a)}}{ 32N^{2}el(\omega a)  {\beta}_{\alpha, \max}}\right\}}\cdot
\end{eqnarray*}
\end{lemma}


\begin{proof}
Considering (\ref{defnumocoamos}) and (\ref{lulu1}) we have, for any finite sequence $\omega_{1}^{j} \in {\cal{A}}^{j}$, 
$$
N_{n}(\omega_{1}^{j})=\sum_{t = 0}^{n-j}\displaystyle\prod_{1\leq i \leq j} \left[{\textbf{1}}_{\left\{X_{t+i}= \omega_{i}\right\}} {\textbf{1}}_{\left\{\xi_{t+i}= 1\right\}} + {\textbf{1}}_{\left\{Y_{t+i}= \omega_{i}\right\}} {\textbf{1}}_{\left\{\xi_{t+i}= 0\right\}} \right].
$$
We define the process {\bf U} by
$$
U_{t}=\displaystyle\prod_{1\leq i \leq j} \left[{\textbf{1}}_{\left\{X_{t+i}= \omega_{i}\right\}} {\textbf{1}}_{\left\{\xi_{t+i}= 1\right\}} + {\textbf{1}}_{\left\{Y_{t+i}= \omega_{i}\right\}} {\textbf{1}}_{\left\{\xi_{t+i}= 0\right\}} \right]- p_{Z}(\omega_{1}^{j})
$$
and we denote by ${\cal{M}}_{i}$ the $\sigma$-algebra generated by $U_{0},...,U_{i}$. Applying Proposition $4$ of Dedecker and Doukhan (2003), we obtain
\begin{eqnarray*}
\left\|N_{n}(\omega_{1}^{j})-(n-j+1)p_{Z}(\omega_{1}^{j})\right\|_{r}&\leq&  \left(2r \sum_{t = 0}^{n-j}\left\|U_{t}\right\|_{r/2}\sum_{k = t}^{l}\left\|\mathbb{E}(U_{k}|{\cal{M}}_{t})\right\|_{\infty}\right)^{1/2},
\end{eqnarray*}
for any $r\geq 2$. We note that $\left\|U_{t}\right\|_{r/2}\leq 1$. Moreover, for any $x_{1}^{t+j}, y_{1}^{t+j}  \in {\cal{A}}_{1}^{t+j}$ and any $\theta_{1}^{t+j}\in \{0,1\}^{t+j}$,  we have
$$
\left\|\mathbb{E}(U_{k}|{\cal{M}}_{t})\right\|_{\infty}=\sup_{x_{1}^{t+j}, \ y_{1}^{t+j}, \ \theta_{1}^{t+j}}\left|\mathbb{P}(Z_{k+j}^{k+1}=\omega_{j}^{1}|X_{1}^{t+j}=x_{1}^{t+j},X_{1}^{t+j}=y_{1}^{t+j}, \xi_{1}^{t+j}=\theta_{1}^{t+j}) - p_{Z}(\omega_{1}^{j})\right|.
$$
Thus, by Lemmas \ref{33gale3} and \ref{lema88888}  we have the following inequality
\begin{equation*}
\left\|N_{n}(\omega)-(n-j+1)p_{Z}(\omega)\right\|_{r} \leq \left[\frac{8r}{(1-\varepsilon)^{l(\omega)}} (n-l(\omega)+1)l(\omega)  {\beta}_{\alpha, \max}\right]^{1/2},
\end{equation*}
where ${\beta}_{\alpha, \max}=\max{\left\{\left( 1 + \frac{\beta_{X}}{\alpha_{X}}\right),\left( 1 + \frac{\beta_{Y}}{\alpha_{Y}}\right)\right\}}$. Thus, as in proof of Lemma \ref{enlema10}, we can write
\begin{eqnarray*}
\mathbb{P}\left( \left|N_{n}(\omega)-(n-l(\omega)+1)p_{Z}(\omega)\right| > t \right) \leq e^{\frac{1}{e}}\exp{\left\{-\frac{t^2(1-\varepsilon)^{l(\omega)}}{ 8e(n-l(\omega)+1)l(\omega)  {\beta}_{\alpha, \max}}\right\}}.
\end{eqnarray*}
Now, one can see that the following inequality holds
\begin{eqnarray}
\lefteqn{\mathbb{P}\left(\left|\frac{N_{n}(\omega a)+1}{N_{n}(\omega .)+|{\cal{A}}|}-\frac{(n-l(\omega))p_{Z}(\omega a)+1}{(n-l(\omega))p_{Z}(\omega)+|{\cal{A}}|}\right| >t'\right)}\nonumber\\
&\leq&\mathbb{P}\left(\left|N_{n}(\omega a)-(n-l(\omega))p_{Z}(\omega a)\right|>\frac{t'}{2}\left[(n-l(\omega))p_{Z}(\omega)+|{\cal{A}}|\right]\right)\nonumber\\
&+& \sum_{b \in {\cal{A}}}\mathbb{P}\left(\left|N_{n}(\omega b)-(n-l(\omega))p_{Z}(\omega b)\right|>\frac{t'}{2|{\cal{A}}|}\left[(n-l(\omega))p_{Z}(\omega)+|{\cal{A}}|\right]\right)\cdot\label{33lem10nancy333}
\end{eqnarray}
Therefore, applying (\ref{33lem10nancy}) we can bound from above (\ref{33lem10nancy333}) by
$$
(N+1)e^{\frac{1}{e}}\exp{\left\{-\left[t-\frac{N+1}{(n-l(\omega))p_{Z}(\omega)}\right]^{2}\frac{{\left[p_{Z}(\omega)\right]}^{2}(n-l(\omega))(1-\varepsilon)^{l(\omega a)}}{ 32N^{2}e(l(\omega a))  {\beta}_{\alpha, \max}}\right\}}\cdot
$$
\end{proof}


\begin{lemma}\label{33enlema11}
For any $\delta >4\left[1+\frac{2(N-1)\beta_{X}}{\min(1,\alpha \beta_{min}^{*})}\right]\varepsilon$, any $\omega \in {\cal{T}}_{X}$, any $u\omega \in \hat{{\cal{T}}}_{n}^{\delta, d}$ and
$$
n>\frac{2(N+1)}{\left\{\delta-4\varepsilon\left[1+\frac{2(N-1)\beta_{X}}{\min{(1,\alpha\beta_{min}^{*})}}\right]\right\}\alpha_{min}^{d}}+d
$$
we have
\begin{equation}\label{cota33enlema11}
\mathbb{P}(\Delta_{n}(u\omega)>\delta)\leq 2N(N+1)e^{\frac{1}{e}}\exp{\left\{-(n-d)\frac{\left[\frac{\delta}{2}-\bar{k}\right]^{2}{\alpha_{min}^{2d}}(1-\varepsilon)^{d+1}}{ 32N^{2}e(d+1)  {\beta}_{\alpha, \max}}\right\}},
\end{equation}
where
$$
\bar{k}=2\varepsilon\left[1+\frac{2(N-1)\beta_{X}}{\min{(1,\alpha\beta_{min}^{*})}}\right]+\frac{N+1}{(n-d)\alpha_{min}^{d}}\cdot
$$
\end{lemma}


\begin{proof}
Similarly to the proof of Lemma \ref{enlema11} we can show that
\begin{eqnarray*}
\lefteqn{\mathbb{P}\left(\Delta_{n}(u\omega)>\delta\right)}\\
&\leq&\sum_{a\in{\cal{A}}}\left[\mathbb{P}\left(\left|\hat{p}_Z(a|u\omega)_n-p_{Z}(a|\mbox{suf}(u\omega))\right|>\frac{\delta}{2}-2\varepsilon\left[1+\frac{2(N-1)\beta_{X}}{\min{(1,\alpha\beta_{\min}^{*})}}\right]\right)\right.\\
&+&\left.\mathbb{P}\left(\left|p_{Z}(a|\mbox{suf}(u\omega))-\hat{p_{Z}}(a|\mbox{suf}(u\omega))_n\right|>\frac{\delta}{2}-2\varepsilon\left[1+\frac{2(N-1)\beta_{X}}{\min{(1,\alpha\beta_{\min}^{*})}}\right]\right)\right].
\end{eqnarray*}
Then, by using Lemma \ref{33lem10} we can bound from above the right side of the last inequality obtaining (\ref{cota33enlema11}).
\end{proof}



\begin{lemma}\label{33lem12}
There exists $d$ such that for any $\delta <D_{d} - 4\varepsilon\left[1 + \frac{2(N-1)\beta_{X}}{\min{(1,\ \alpha\beta_{\min}^{*})}}\right],$ any $\omega\in\hat{{\cal{T}}}_{n}^{\delta,d}$, with $l(\omega)<K,$ $\omega \notin {\cal{T}}_{X}$, and any
$$
n>\frac{2(N+1)}{\left\{D_{d} - \delta -4\varepsilon\left[1 + \frac{2(N-1)\beta_{X}}{\min{(1,\alpha\beta_{\min}^{*})}}\right]\right\}\alpha^{d}}+d
$$
we have
\begin{equation*}
\mathbb{P}\left(\bigcap_{u\omega \in T\left.\right|_{d}}\left\{\Delta_{n}(u\omega)\leq \delta\right\}\right)\leq 2(N+1)e^{\frac{1}{e}}\exp\left[-(n-d)\frac{\left[D_{d} - \delta-\bar{k}\right]^{2}{\alpha_{\min}^{2d}(1-\varepsilon)^{d}}}{128N^{2}e(d+1){\beta}_{\alpha, \max}}\right]
\end{equation*}
where $\bar{k}$ is as in Lemma \ref{33enlema11}.
\end{lemma}


\begin{proof}
Take
$$
d=\underset{u\notin {\cal{T}}_{X}, \ l(u)<K}{\max}\min{\left\{k:\mbox{ there exists } \omega \in C_{k} \ \mbox{with}  \ \omega\succ u\right\}}.
$$
As in proof of Lemma \ref{enlema12345} we can show that
\begin{eqnarray*}
\lefteqn{\mathbb{P}\left(\Delta_{n}(\bar{u\omega})\leq\delta\right)}\\
&\leq& \mathbb{P}\left(\bigcap_{a\in {\cal{A}}}\left\{\left|\hat{p}_Z(a|\bar{u\omega})_n-p_{Z}(a|\bar{u\omega})\right|\geq\frac{D_{d} - \delta-4\varepsilon\left[1 + \frac{2(N-1)\beta_{X}}{\min{(1,\alpha\beta_{min}^{*})}}\right]}{2}\right\}\right)\\
&+&\mathbb{P}\left(\bigcap_{a\in {\cal{A}}}\left\{\left|\hat{p}_Z(a|\bar{u\omega})_n-p_{Z}(a|\bar{u\omega})\right|\geq\frac{D_{d} - \delta - 4\varepsilon\left[1 + \frac{2(N-1)\beta_{X}}{\min{(1,\alpha\beta_{\min}^{*})}}\right]}{2}\right\}\right).\\
\end{eqnarray*}
Therefore, if $\delta < D_{d} - 4\varepsilon\left[1 + \frac{2(N-1)\beta_{X}}{\min{(1,\alpha\beta_{\min}^{*})}}\right]$ and
$$
n>\frac{2(N+1)}{\left\{D_{d} - \delta -4\varepsilon\left[1 + \frac{2(N-1)\beta_{X}}{\min{(1,\alpha\beta_{\min}^{*})}}\right]\right\}\alpha^{d}}+d
$$
we can use Lemma  \ref{33lem10} to conclude the proof of this lemma.
\end{proof}


\begin{33pff}
{\rm By proceeding as in the proof of Theorem \ref{enteorema2} we can show that
$$
\mathbb{P}\left(\hat{{\cal{T}}}_{n}^{\delta,d}\left.\right|_{K}\neq {\cal{T}}_{X}\left.\right|_{K}\right)\leq\displaystyle\sum_{\underset{l(\omega)<K}{\omega \in {\cal{T}}_{X}}}\displaystyle\sum_{u\omega \in \hat{{\cal{T}}}_{3n}^{\delta, d}}\mathbb{P}\left(\Delta_{n}(u\omega)>\delta\right)+\displaystyle\sum_{\underset{l(\omega)<K}{\omega \in \hat{{\cal{T}}}_{n}^{\delta,d}}}\mathbb{P}\left(\bigcap_{u\omega \in {\cal{T}}_{X}\left.\right|_{d}}\left\{\Delta_{n}(u\omega)\leq\delta\right\}\right)
$$
if $d < n$. Therefore, by means of Lemmas \ref{33enlema11} and \ref{33lem12} we obtain 
$$
\mathbb{P}\left(\hat{{\cal{T}}}_{n}^{\delta, d}\left|\right._{K}\neq {\cal{T}}_{X}\left|\right._{K}\right)\leq c_{2}exp\left\{-c_{3}(n-d)\right\},
$$
where $c_{2}=48N^{d}(N+1)e^{\frac{1}{e}}$ and $c_{3}=\frac{\left[\min\left(D_{d} - \delta,\delta\right)-2\bar{k}\right]^{2}{\alpha^{2d}}}{128N^{2}e(d+1){\beta}_{\alpha, \ max}}$.}
\end{33pff}


\begin{333pfc}
{\rm It follows from Theorem \ref{3enteo2}, First Borel-Cantelli Lemma and the fact that the quotas for the estimation error of the truncated context tree are summable at $n$ for appropriate choices of $d$ and $\delta$.}
\end{333pfc}

\section{Comparisons}\label{comparacao}

In this section we compare the results obtained in this paper with the
corresponding ones presented in Collet, Galves and Leonardi
(2008). For this purpose, let {\bf X} and {\bf Y} be independent processes taking values on the
alphabet ${\cal{A}}=\left\{0,1\right\}$. Furthermore, we assume that these
processes are non-null and have summable continuity rate with the same constants $\alpha_{X}$ and $\beta_{X}$.

To compare the bounds we are going to couple the processes using the
same Bernoulli sequence  ${\boldsymbol \xi}$ independent of the processes
{\bf X} and {\bf Y}, with $\mathbb{P}\left(\xi_{t}=1\right)=1-\varepsilon$,
where $\varepsilon$ is fixed  in $(0,1)$. 
Now, we define the stochastically perturbed chains ${\bf Z_{1}}$, ${\bf Z_{2}}$ and ${\bf Z_{3}}$ by
\begin{eqnarray}\label{kcomp1}
Z_{1,t}&=& X_t + (1-\xi_t) \, \mbox{(mod 2)},\\
Z_{2,t}&=& X_t \cdot \xi_t,\label{kcomp2}\\
Z_{3,t} &=& \left\{\begin{array}{rl}
X_t, &\mbox{if}\ \xi_t=1,\\
Y_t,& \mbox{if}\ \xi_t=0 ,\\
\end{array}
\right.\label{kcomp3}
\end{eqnarray}
where $t \in {\mathbb Z}$. The model (\ref{kcomp1}) was proposed in Collet,  Galves  and Leonardi (2008). We assume that ${\bf Z_{i}}$ is compatible with $q_{i}(\cdot|\cdot)$ the law of the process, for $i=1,2 \ \mbox{or} \ 3$.

In the model (\ref{kcomp1}), the process ${\bf Z_{1}}$ will be different from the process {\bf X} whenever $\xi_{t}=0$, which occurs with probability $\varepsilon$. The process ${\bf Z_{2}}$, defined by (\ref{kcomp2}) is equal to the process {\bf X} with high probability $1-\varepsilon$.  The third contamination model, defined in (\ref{kcomp3}), is such that at each instant of time the process ${\bf Z_{3}}$ either is equal to the process {\bf X}, with high probability $1-\varepsilon$, or is equal to the process {\bf Y}, with small probability $\varepsilon$.

Theorem $1$ of  Collet,  Galves  and Leonardi (2008) states that 
\begin{equation*}
k_{1}:= \varepsilon\left[1 + \frac{4\beta_{X}}{\min{(1,\alpha_{X}\beta_{X}^{*})}}\right]\geq\sup_{a\in {\cal{A}}, \ \omega_{-k}^{-1} \in {\cal{A}}_{-k}^{-1}}\left|q_{1}\left(a|\omega_{-k}^{-1}\right)-p_{X}\left(a|\omega_{-k}^{-1}\right) \right|,
\end{equation*}
where $\varepsilon \in (0,1)$, $k\geq 0$ and  $\beta_{X}^{*}=\prod_{k=0}^{+\infty}(1-\beta_{k,X})<+\infty$. Our Theorem \ref{enteor1} says that
\begin{equation*}
k_{2}:=\varepsilon\left[1 + \frac{4\beta_{X}}{\min{(1,(1+\varepsilon)\alpha_{X}\beta_{X}^{*})}}\right]\geq\sup_{a\in {\cal{A}}, \ \omega_{-k}^{-1} \in {\cal{A}}_{-k}^{-1}}\left|q_{2}\left(a|\omega_{-k}^{-1}\right)-p_{X}\left(a|\omega_{-k}^{-1}\right) \right|.
\end{equation*}
Theorem \ref{enteor31} says that  
\begin{equation*}
k_{3}:=\varepsilon\left[2 + \frac{4\beta_{X}}{\min{(1,\alpha\beta^{*}_{\min})}}\right]\geq\sup_{a\in {\cal{A}}, \ \omega_{-k}^{-1} \in {\cal{A}}_{-k}^{-1}}\left|q_{3}\left(a|\omega_{-k}^{-1}\right)-p_{X}\left(a|\omega_{-k}^{-1}\right) \right|,
\end{equation*}
where $\alpha\beta^{*}_{\min}=\min\{\alpha_{X}\beta_{X}^{*},\alpha_{X}\}$. It is easy to see that
\begin{equation}\label{jkmocomp}
k_{2}\leq k_{1}\leq k_{3}.
\end{equation}

We observe that the inequality $k_{2}\leq k_{1}$ was expected by
definitions of the chains ${\bf Z_{1}}$ and ${\bf Z_{2}}$. To see this
we note that in the model (\ref{kcomp1}) it is possible to change both
symbols of the original process {\bf X}  by the Bernoulli effect
${\boldsymbol \xi}$ whereas  in the model (\ref{kcomp2}) only the symbol $1$ of the process {\bf X} has positive probability of being modified by process ${\bf\xi}$. Since the process ${\bf Z_{3}}$ is more general than the other two then the inequality $k_{1}\leq k_{3}$ was also expected. 

\section*{Acknowledgements}

This is part of the Ph.D. Thesis suported by FAPESP fellowship (grant
2008/10693-5). N. L. Garcia was partially supported by CNPq grants
475504/2008-9, 302755/2010-1 and 476764/2010-6

\end{document}